\documentclass[a4paper]{article}
\usepackage{amsmath}
\usepackage{amsfonts}
\usepackage{amsthm}
\usepackage{amssymb}
\usepackage[toc,page]{appendix}
\usepackage{array}
\usepackage{enumitem}
\usepackage{float}
\usepackage{setspace}
\usepackage{graphicx}
\usepackage{leftidx}
\usepackage{arydshln}
\usepackage{tikz-cd} 
\usepackage{hyperref}
\usepackage{mathrsfs}
\usepackage{etoolbox}
\usepackage{makeidx}
\usepackage{tikz}
\usepackage{soul}
\usepackage{mathtools}
\usetikzlibrary{decorations.pathmorphing}

\patchcmd{\thebibliography}{\section*{\refname}}{}{}{}

\newtheorem{Theorem}{Theorem}[section]
\newtheorem{Proposition}[Theorem]{Proposition}
\newtheorem{Lemma}[Theorem]{Lemma}
\newtheorem{Definition}[Theorem]{Definition}
\newtheorem{Meta-Definition}[Theorem]{Meta-Definition}
\newtheorem{Corollary}[Theorem]{Corollary}
\newtheorem{Remark}[Theorem]{Remark}
\newtheorem{Example}[Theorem]{Example}
\newtheorem{Note}[Theorem]{Note}

\DeclareMathOperator{\re}{Re}

\DeclareMathOperator{\id}{id}
\DeclareMathOperator{\Ad}{Ad}
\DeclareMathOperator{\im}{Im}
\DeclareMathOperator{\TO}{\underline{\mathbf{TO}}}
\DeclareMathOperator{\FROM}{\underline{\mathbf{FROM}}}

\DeclareMathOperator{\Arch}{Arch}

\DeclareMathOperator{\Aff}{\mathcal{A}ff}

\title{Natural levels in return maps of elementary polycycles}

\author{Melvin Yeung}

\date{}

\begin{document}
	
	\maketitle
	
	\begin{abstract}
		We will provide a proof of a known specific case of Dulac's Theorem in the style of Ilyashenko. From this we derive a quasi-analyticity result for some return maps of polycycles and we give a Structural Theorem for the formal asymptotics of such a polycycle.
	\end{abstract}
	
	\section{Introduction}
	
	\subsection{Goal}
	
	The aim of this article is to show a general reduction contained in \cite{ilyashenkoFiniteness} originally intended to prove:
	
	\begin{Theorem}[Dulac (Ilyashenko, Ecalle)]
		Any real planar polynomial vector field has at most a finite amount of limit cycles, i.e. isolated periodic orbits.
	\end{Theorem}
	
	Then we will be apply this formalism to prove a known case of the Theorem of Dulac (see \cite{Moussu1991} and \cite{Kaiser2017AnalyticCO}).\\
	
	For a more historic overview of this problem and the broader 16th problem of Hilbert, see \cite{Ilyashenko2002CentennialHO}. The main difficulty of the Theorem of Dulac is the following:
	
	\begin{Theorem}
		Let $\Gamma$ be a real analytic polycycle in a real analytic vector field on a real analytic $2$-manifold, then there exists an open $U$ containing $\Gamma$, but no limit cycles. 
	\end{Theorem}
	
	The way this is tackled is using so-called monodromy maps and desingularization. For definitions of polycycles and monodromy maps, see \cite[Chapter IV, part 24C]{Ilyashenko08lectureson}. The theorem above is easily implied by:
	
	\begin{Theorem}\label{ThmFinLim}
		Let $\Delta_{\Gamma}$ be a monodromy map of a polycycle with only elementary equilibria (i.e. ones with at least one non-zero eigenvalue), then either:
		
		\begin{enumerate}
			\item{$\Delta_{\Gamma} - \id \equiv 0$}
			\item{$\Delta_{\Gamma} - \id$ has no zeroes close enough to the origin.}
		\end{enumerate}
	\end{Theorem}
	
	However this Theorem has both easy and hard cases, depending on something called the depth of the polycycle:\\
	
	So take a polycycle homeomorphic to a circle in a vector field on a real analytic $2$-manifold, of the form described above.\\
	
	We may now parametrize our polycycle, say $\Gamma$, with $\gamma: [0, 1] \to \Gamma$, starting at an arbitrary point $x \in \Gamma$, say that $x$ is not equal to an equilibrium. Suppose that $\gamma$ is injective on $(0, 1)$.\\
	
	Then for a $t \in [0, 1]$ we can define the depth of $\gamma$ at $t$, $D(\gamma, t)$ as follows:
	
	\[D(\gamma, t) \coloneqq \# \{\text{semihyperbolic saddles in }\gamma((0, t]) \text{ from the centre direction}\}\] 
	\[- \# \{\text{semihyperbolic saddles in }\gamma((0, t]) \text{ from the hyperbolic direction}\}\]~\\
	
	Note first that this is well-defined because a polycycle only has a transit map along a 'single side'.\\
	
	Using $D$ we can get the trivial cases of Theorem \ref{ThmFinLim}:
	
	\begin{Theorem}
		If $D(\gamma, 1) \neq 0$, then the monodromy map of $\Gamma$ is either flat or inverse to flat, so there are only a finite amount of limit cycles.
	\end{Theorem}
	
	A proof follows easily from normal form theory, for example a normal form such as \cite[Chapter IV, part 22H, 22.23]{Ilyashenko08lectureson} implies that the transit map of a semihyperbolic saddle is roughly $e^{-\frac{1}{z}}$ or its inverse, resulting in the Theorem. From this $D(\gamma, t)$ we also get our interesting case:
	
	\begin{Definition}
		We call $\Gamma$ balanced if $D(\gamma, 1) = 0$.\\
		
		A parametrization obviously exists such that $D(\gamma, t)$ is always $\leq 0$, in that case we call the minimum of $D(\gamma, t)$ the depth of $\Gamma$.
	\end{Definition}
	
	The point of the general procedure of \cite{ilyashenkoFiniteness} is essentially that the monodromy map of a balanced polycycle decomposes in a natural way into a finite sum of essentially different sizes. Somewhat similarly to the cryptolinear formulas found in \cite[Chapter 5]{ecalleconstructiveproof}. For a concrete formulation see \ref{ThmADT}.\\
	
	A polycycle consists of a sequence of saddles which are either hyperbolic or semihyperbolic. Any semihyperbolic saddle (see e.g. \cite[I, 4]{Ilyashenko08lectureson}) is formally equivalent to a semihyperbolic saddle of the form:
	
	\[\frac{x^{k + 1}}{1 + a x^{k}}\frac{\partial}{\partial x} - y\frac{\partial}{\partial y}\]
	
	\begin{Definition}\label{DefConvSad}
		We call a semihyperbolic saddle convergent if there exists a real analytic function (analytic in zero) conjugating the semihyperbolic saddle to that of its formal normal form above.\\
		
		Analogously for hyperbolic saddles. I.e. We call a hyperbolic saddle convergent is there exists a real analytic function (analytic in zero) conjugating the hyperbolic saddle to one of the following:
		
		\[x\frac{\partial}{\partial x} - \lambda y\frac{\partial}{\partial y}\]
		
		\[x\frac{\partial}{\partial x} - y\left(\frac{m}{n} \pm x^{mk}y^{nk} \pm a x^{2km}y^{2kn}\right)\frac{\partial}{\partial y}\]
		
		For some $\lambda, a \in \mathbb{R}, m, n \in \mathbb{N}$, $\gcd(m, n) = 1$ and supposing the $\pm$ align, i.e. they are either both $+$ or both $-$.\\
		
		We call a polycycle convergent if all the saddles it contains are convergent.
	\end{Definition}
	
	\begin{Remark}
		Note that this is not only a restriction on the semihyperbolic saddles but also on the hyperbolic saddles. Remember (see e.g. \cite[Chapter 24]{Ilyashenko08lectureson}) that transit maps of hyperbolic saddles are given in the logarithmic chart $\zeta = -\ln(z)$ by almost regular maps, i.e. analytic maps on standard quadratic domains with some associated asymptotic Dulac series.\\
		
		In some sense what we will ask for with convergent hyperbolic saddles is that this Dulac series converges pointwise on some halfplane $\{\re(z) \geq a\}$ for some $a \in \mathbb{R}$. See e.g. \cite[p. 18]{ilyashenkoFiniteness} for a sketch of proof that the transit map of the normal forms is pointwise convergent. 
	\end{Remark}
	
	From here we get the case of the Theorem of Dulac we actually want to prove in this article:
	
	\begin{Theorem}\label{ThmConvPoly}
		Let $\Delta$ be the monodromy map of a balanced convergent polycycle, then either:
		
		\begin{enumerate}
			\item{$\Delta - \id \equiv 0$}
			\item{$\Delta - \id$ has no zeroes close enough to the origin.}
		\end{enumerate}
	\end{Theorem}
	
	The advantage of this method over the ones mentioned in \cite{Kaiser2017AnalyticCO} and \cite{Moussu1991} is in essence an explicit quasi-analyticity estimate, which is new to this approach:\\
	
	To state this let us start with introducing a new coordinate, the logarithmic chart, which essentially unwinds the branch cut of the logarithm necessary to talk about maps such as $\frac{1}{-\ln(z)}$, it being the inverse of $e^{-\frac{1}{z}}$:
	
	\begin{Definition}
		Let $\dot{\mathbb{C}}$ be $\mathbb{C}$ with a logarithmic branch cut at the origin, let $g: (\dot{\mathbb{C}}, 0) \to (\dot{\mathbb{C}}, 0)$ be a germ of a map, then we define $g^{\log}$ to be $g$ in the new coordinate:
		
		\[\zeta = -\ln(z)\]
		
		more concretely:
		
		\[g^{\log}(\zeta) = -\ln\left(g\left(e^{-\zeta}\right)\right)\]
	\end{Definition}
	
	Then
	
	\begin{Proposition}\label{PropQAEst}
		Let $\Delta$ be the monodromy map of a balanced convergent polycycle of depth $\leq n$, suppose that for all $\lambda, \mu > 0$, for $\xi$ large enough real:
		
		\[|\Delta^{\log}(\xi) - \xi| < e^{-\lambda \exp^{n}(\mu \xi)}\]
		
		($n$-fold composition) then $\Delta \equiv \id$.
	\end{Proposition}
	
	\begin{Note}
		While this has not been done in \cite{Kaiser2017AnalyticCO} it is no doubt possible with careful study of the appearing transseries, however here we give a geometric reason which can not be captured by \cite{Kaiser2017AnalyticCO} because it requires techniques to carefully control domains of functions involved.
	\end{Note}
	
	\subsection{Notations and strategy of proof}
	
	First of all we make the convention that $f^{k}(x)$ or $f^{k}$ is the $k$-fold composition of $f$, $f(x)^{k}$ or $(f)^{k}$ is $f$ multiplied pointwise with itself $k$ times.\\
	
	Then we will need a few small baseline definitions. 
	
	\begin{Definition}
		We define the class $\Aff$ to be the class of affine real analytic functions with positive derivative, i.e. the functions:
		
		\[\zeta \mapsto \alpha\zeta + \beta \quad,\quad \alpha, \beta \in \mathbb{R}, \alpha > 0\]
	\end{Definition}
	
	\begin{Definition}
		Let $g, h$ be maps, then we define:
		
		\[\Ad(g)(h) = g^{-1}\circ h \circ g\]
		
		And particularly we define:
		
		\[A(h) = \Ad(\exp)(h) = \ln \circ h \circ \exp\]
		
		Given germs of real functions $(g_{i})_{i \in I}$ near infinity, then $\langle g_{i}, i \in I\rangle$ is the group generated by the $g_{i}$ under composition (when this makes sense).
	\end{Definition}
	
	This formalism then depends on two classes of functions:
	
	\begin{Definition}
		Let $\mathcal{R}$ and $\mathcal{NC}$ be two sets of germs of real functions near $+\infty$. We say that they are admissible if:
		
		\begin{enumerate}
			\item{For each germ of an analytic function $f$ at $0$, $f^{\log}\in \mathcal{R}$.}
			\item{$\Aff \subset \mathcal{R}$.}
			\item{For each $\beta \in \mathbb{R}$ the function given by:
				
				\[\zeta \mapsto \zeta + \ln\left(1 + \beta e^{\zeta}\right)\]
				
				is in $\mathcal{R}$.}
			\item{For all elements $f$ of $\mathcal{R}$, either $f \in \Aff$ or there exists $\lambda > \mu > 0$ and an $a \in \Aff$ such that for all $x$ large enough:
				
				\[e^{-\lambda x} < |f(x) - a(x)| < e^{-\mu x}\]
				
				The set of $\mathcal{R}$ for which this $a$ is the identity together with the identity itself is denoted $\mathcal{R}^{0}$.}
			\item{The same holds for $\mathcal{NC}$, but $a$ is always the identity and $\mathcal{NC} \cap \Aff = \{\id\}$.}
			\item{$\mathcal{R}$ forms a group under composition.}
		\end{enumerate}
		
		We call a hyperbolic saddle $(\mathcal{R}, \mathcal{NC})$-admissible along a given quadrant if there exists some analytic sections with respect to which the Dulac map of the saddle in the logarithmic chart is given by an element of $\mathcal{R}$.\\
		
		We call a semihyperbolic saddle $(\mathcal{R}, \mathcal{NC})$-admissible along a given quadrant if there exists some analytic sections with respect to which the Dulac map of the saddle in the logarithmic chart is given by an element of:
		
		\[\TO_{(\mathcal{R}, \mathcal{NC})} \coloneqq \mathcal{R}^{0} \circ \exp \circ \mathcal{R}^{0} \circ \mathcal{NC}\circ \Aff\]
		
		We call a polycycle $(\mathcal{R}, \mathcal{NC})$-admissible if every saddle in it is $(\mathcal{R}, \mathcal{NC})$-admissible.\\
		
		When it is clear from context we will drop $(\mathcal{R}, \mathcal{NC})$ from the notation.
	\end{Definition}
	
	\begin{Remark}
		This is an abstract generalization of the framework used in \cite{ilyashenkoFiniteness} in order to prove finiteness for polycycles. The idea being that if you understand normalization of certain classes of saddles, you will be able to use that to prove finiteness theorems for classes of polycycles with those saddles. For the original version of this, see Subsection \ref{SubsecOrigin}.\\
		
		In this article we will apply the framework to the set of convergent saddles and convergent saddle nodes.
	\end{Remark}
	
	All our bookkeeping will then be summed up in what we call the sufficiency of axioms. Because the full statement is rather long and terse we will instead write here a shortened version and write about the intent of the axioms. Later we will write a full version.
	
	\begin{Theorem}[Sufficiency of axioms]\label{ThmSuffAx}
		Suppose for an admissible choice of $(\mathcal{R}, \mathcal{NC})$ there exists subsets $\mathcal{FC}_{i}^{n}$ of germs of real analytic functions at infinity satisfying the following axioms:
		
		\begin{enumerate}
			\item{Axiom A1: For all $n$ and $i = 0, 1$, for any non-zero element $\alpha \in \mathcal{FC}^{n}_{i}$ there exist $\lambda < \mu < 0$ such that:
				
				\[e^{\lambda x} < \alpha < e^{\mu x}\]}
			\item{Axiom A2: For all $n$:
				\begin{enumerate}
					\item[A2a]{$A^{n}\mathcal{R}^{0} \subset \id + \mathcal{F}^{n}_{0\id}$}
					\item[A2b]{$A^{n}\mathcal{NC} \subset \id + \mathcal{F}^{n}_{1\id}$}
			\end{enumerate}}
			\item{Axioms $A3-A8$ implying that $\mathcal{FC}^{n}_{i}$ is closed under certain constructions and operations.}
		\end{enumerate}
		
		Then all $(\mathcal{R}, \mathcal{NC})$-admissible polycycles have a (one-sided) neighbourhood without limit cycles.
	\end{Theorem}
	
	Roughly speaking, the idea is to rearrange the elements appearing in the transit map of the polycycle. The picture looks particularly nice when all semihyperbolic saddles are in their formal normal form with formal constant $0$ and $k = 1$. The transit map of a polycycle can be represented by a diagram as follows:
	
	\begin{figure}[H]
		\begin{tikzpicture}
			\draw[help lines, color=gray!30, dashed] (0,-4.9) grid (9.9,1);
			\draw[->,ultra thick] (0,0)--(10,0) node[right]{$x$};
			\draw[->,ultra thick] (0,-5)--(0,1) node[above]{$y$};
			\draw[blue](0, 0)--(1, -1);
			\draw[blue](1, -1)--(2, -2);
			\draw[blue](2, -2)--(3,-1);
			\draw[blue](3, -1)--(4, -1);
			\draw[blue](4, -1)--(5, -2);
			\draw[blue](5, -2)--(6, -3);
			\draw[blue](6, -3)--(7, -3);
			\draw[blue](7, -3)--(8, -2);
			\draw[blue](8, -2)--(9, -1);
			\draw[blue](9, -1)--(10, 0);
		\end{tikzpicture}
		
		\caption{Depth of a polycycle}
	\end{figure}
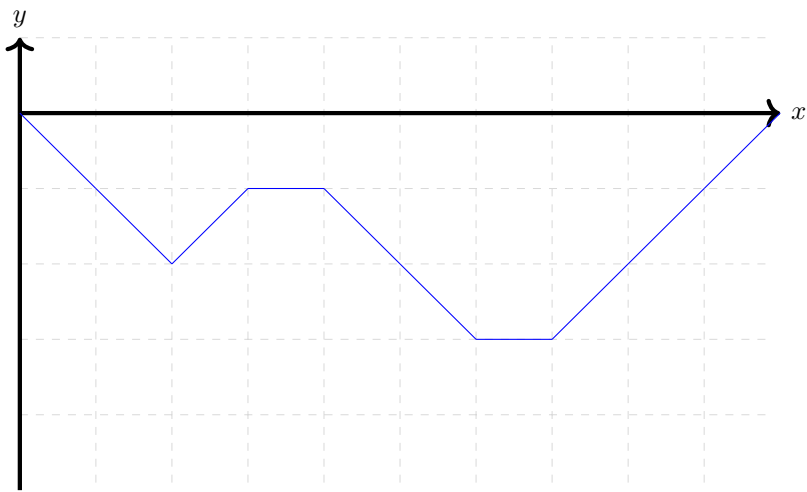
	
	We go down by one when we encounter a semihyperbolic saddle going to the centre manifold, we go up by one when we encounter a semihyperbolic saddle going from the centre manifold and we stay constant when we encounter a hyperbolic saddle. In the case mentioned above, going down means a pure $\exp$ (in the logarithmic chart) and going up means a pure $\ln$. The idea put forward in the Structural Theorem will be to add the appropriate amount of $\exp$ and $\ln$ in-between to essentially add the following lines to the diagram:
	
	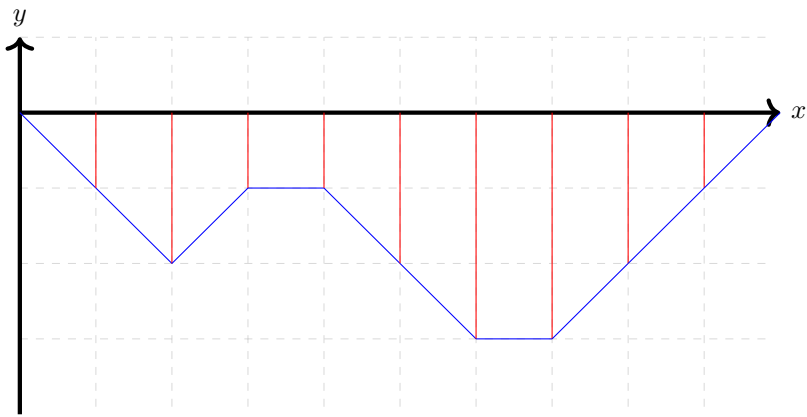
\begin{figure}[H]
		\begin{tikzpicture}
			\draw[help lines, color=gray!30, dashed] (0,-3.9) grid (9.9,1);
			\draw[->,ultra thick] (0,0)--(10,0) node[right]{$x$};
			\draw[->,ultra thick] (0,-4)--(0,1) node[above]{$y$};
			\draw[blue](0, 0)--(1, -1);
			\draw[red](1, -1)--(1, 0);
			\draw[blue](1, -1)--(2, -2);
			\draw[red](2, -2)--(2, 0);
			\draw[blue](2, -2)--(3,-1);
			\draw[red](3, -1)--(3, 0);
			\draw[blue](3, -1)--(4, -1);
			\draw[red](4, -1)--(4, 0);
			\draw[blue](4, -1)--(5, -2);
			\draw[red](5, -2)--(5, 0);
			\draw[blue](5, -2)--(6, -3);
			\draw[red](6, -3)--(6, 0);
			\draw[blue](6, -3)--(7, -3);
			\draw[red](7, -3)--(7, 0);
			\draw[blue](7, -3)--(8, -2);
			\draw[red](8, -2)--(8, 0);
			\draw[blue](8, -2)--(9, -1);
			\draw[red](9, -1)--(9, 0);
			\draw[blue](9, -1)--(10, 0);
		\end{tikzpicture}
		
		\caption{Splitting up the map}
	\end{figure}
	
	The idea is to rearrange all these blocks, note that the fest block is non-trivial because it still has the transit map between saddles. The way we rearrange is to put the blocks with smallest height first and then go to blocks with larger and larger height, so this particular map would rearrange as follows:
	
	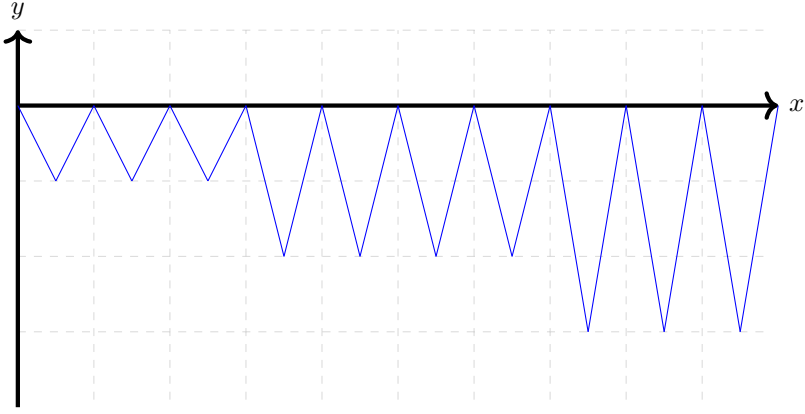
\begin{figure}[H]
		\begin{tikzpicture}
			\draw[help lines, color=gray!30, dashed] (0,-3.9) grid (9.9,1);
			\draw[->,ultra thick] (0,0)--(10,0) node[right]{$x$};
			\draw[->,ultra thick] (0,-4)--(0,1) node[above]{$y$};
			\draw[blue](0, 0)--(0.5, -1);
			\draw[blue](0.5, -1)--(1, 0);
			\draw[blue](1, 0)--(1.5, -1);
			\draw[blue](1.5, -1)--(2, 0);
			\draw[blue](2, 0)--(2.5, -1);
			\draw[blue](2.5, -1)--(3, 0);
			\draw[blue](3, 0)--(3.5, -2);
			\draw[blue](3.5, -2)--(4, 0);
			\draw[blue](4, 0)--(4.5, -2);
			\draw[blue](4.5, -2)--(5, 0);
			\draw[blue](5, 0)--(5.5, -2);
			\draw[blue](5.5, -2)--(6, 0);
			\draw[blue](6, 0)--(6.5, -2);
			\draw[blue](6.5, -2)--(7, 0);
			\draw[blue](7, 0)--(7.5, -3);
			\draw[blue](7.5, -3)--(8, 0);
			\draw[blue](8, 0)--(8.5, -3);
			\draw[blue](8.5, -3)--(9, 0);
			\draw[blue](9, 0)--(9.5, -3);
			\draw[blue](9.5, -3)--(10, 0);
		\end{tikzpicture}
		
		\caption{Rearranging the map}
	\end{figure}
	
	The way we do this is by conjugation, so if we had an $f$ which is a stand in for a spike of height $2$ and an $g$ for a spike of height $3$ and they occurred in the order $g \circ f$, i.e. the wrong order, then we would rewrite this as $f \circ (f^{-1} \circ g \circ f)$ and we would want $f^{-1} \circ g \circ f$ to have the same behaviour as another spike of height $3$. Then once the spikes are rearranged like this the spikes of the same height can also be ordered by conjugation. Then we Taylor expand out, so say we had $f \circ g$ (so the right order), then we write $f = \id + \delta$ and $g = \id + \epsilon$ and we Taylor expand out to:
	
	\[f \circ g = \id + \delta + \epsilon + \tilde{\epsilon}\]
	
	With $\tilde{\epsilon}$ containing the higher orders of the Taylor expansion. Now let us take a look at what this height means and we claim that $\delta$ should be much larger than both $\epsilon$ and $\tilde{\epsilon}$. Let us consider a function $\id + \beta$ with $\beta$ small, then by Taylor expansion:
	
	\[\Ad(g)(\id + \beta) = \id + ((g^{-1})' \circ g)\cdot \beta\circ g + O((\beta \circ g)^{2})\]
	
	In particular a map of depth $n$ can be written roughly as:
	
	\[A^{n}(\id + \beta) = \id + \left(\prod_{k = 1}^{n}\frac{1}{\exp^{k}}\right)\cdot \beta\circ \exp^{n} + O((\beta\circ \exp^{n})^{2})\]
	
	Remembering that $\exp$ in the logarithmic chart is just the standard flat function, this means that the depth of a spike corresponds to the level of flatness. So when we Taylor expand out we get on different levels contributions of different levels of flatness, so they do not interfere with each other. In the proof of the sufficiency of axioms we will use this to decompose any transit map as a sum of terms which are `of different sizes' (the Additive Decomposition Theorem) and then we use the usual argument saying that if we look at the difference with the identity we either have that it is identically zero, or this decomposition will have a leading term, preventing limit cycles close to the polycycle (on one side).\\
	
	\begin{Remark}\label{RemA1Bounds}
		This is also why axiom $A1$ contains both a lower and upper bound. Clearly the lower bound is needed for non oscillation, but the upper and lower bounds also make sure that the 'different levels of the spikes are separated'. More concretely a straightforward calculation shows that if $A1$ holds, then for $g, h \in G^{n - 1}$:
		
		\[\mathcal{FC}^{n}_{i} \circ \exp^{n} \circ g = \mathcal{FC}^{n}_{i} \circ \exp^{n} \circ h \Rightarrow c < \frac{\exp^{n}(g(x))}{\exp^{n}(h(x))} < C\]
		
		for some $x$ large enough and $c, C > 0$. The inverse implication follows by axiom $A4a$ (see Theorem \ref{ThmSuffAxFull}).
	\end{Remark}
	
	We have two principal results:
	
	\begin{enumerate}
		\item{Sufficiency of axioms.}
		\item{The convergent case.}
	\end{enumerate}
	
	We will give each of these results their own section.\\
	
	Afterwards we will collect some remarks regarding what is still left in \cite{ilyashenkoFiniteness}.\\
	
	\section{Sufficiency of axioms}
	
	For this entire section we fix a choice of admissible $(\mathcal{R}, \mathcal{NC})$. We will not talk about multiple admissible choices so in all cases we will omit any extra notation.
	
	\subsection{The Structural Theorem}
	
	The goal of this part is to prove:
	
	\begin{Theorem}[Structural Theorem]
		Let $G_{n}$ be the set of monodromy maps of depth $n$ balanced admissible polycycles. We have:
		
		\[G_{n}^{\log} \subset \langle \Aff, A^{i}\mathcal{R}^{0}, A^{j}\mathcal{NC}\mid 0 \leq i \leq n, 0 \leq j \leq n - 1\rangle\]
	\end{Theorem}
	
	We will split up the full structural Theorem into two Lemma's:
	
	\begin{Lemma}
		We have:
		
		\[G_{n}^{\log} \subset \langle A^{k}(\FROM\circ \mathcal{R}\circ \TO), \mathcal{R} \mid 0 \leq k \leq n - 1\rangle\]
		
		with $\FROM$ defined as the inverse maps of $\TO$.
	\end{Lemma}
	
	\begin{proof}
		A transit map of a polycycle can be decomposed as:
		
		\begin{enumerate}
			\item{Analytic transit maps between different equilibria, which in the logarithmic chart are in $\mathcal{R}$.}
			\item{Transit maps of hyperbolic saddles which in the logarithmic chart are in $\mathcal{R}$.}
			\item{Transit maps of semihyperbolic saddles which in the logarithmic chart are in either $\FROM$ or $\TO$.}
		\end{enumerate}
		
		The point is to put 'the right amount' of $\exp$ and $\ln$ between any two of the above classes, which depends on the depth at the moment. What the 'right amount' is, is perhaps best demonstrated visually using a diagram from \cite{ilyashenkoFiniteness}. So, assuming we start at a admissible point in the polycycle, i.e. such that the depth stays negative, we can make a diagram of the functions we encounter in order, roughly graphing out the depth. Going from left to right, a horizontal line represents an element of $\mathcal{R}$, a diagonally downward line represents an element of $\TO$ and a diagonally upward line represents an element of $\FROM$. Then one for example has a diagram as follows (note that this does not actually represent the maps of a polycycle because we ignore maps of class 1):\\
		
		\begin{figure}[H]
			\begin{tikzpicture}
				\draw[help lines, color=gray!30, dashed] (0,-4.9) grid (9.9,1);
				\draw[->,ultra thick] (0,0)--(10,0) node[right]{$x$};
				\draw[->,ultra thick] (0,-5)--(0,1) node[above]{$y$};
				\draw[blue](0, 0)--(1, -1);
				\draw[blue](1, -1)--(2, -2);
				\draw[blue](2, -2)--(3,-1);
				\draw[blue](3, -1)--(4, -1);
				\draw[blue](4, -1)--(5, -2);
				\draw[blue](5, -2)--(6, -3);
				\draw[blue](6, -3)--(7, -3);
				\draw[blue](7, -3)--(8, -2);
				\draw[blue](8, -2)--(9, -1);
				\draw[blue](9, -1)--(10, 0);
			\end{tikzpicture}
			
			\caption{Depth of a polycycle}
		\end{figure}
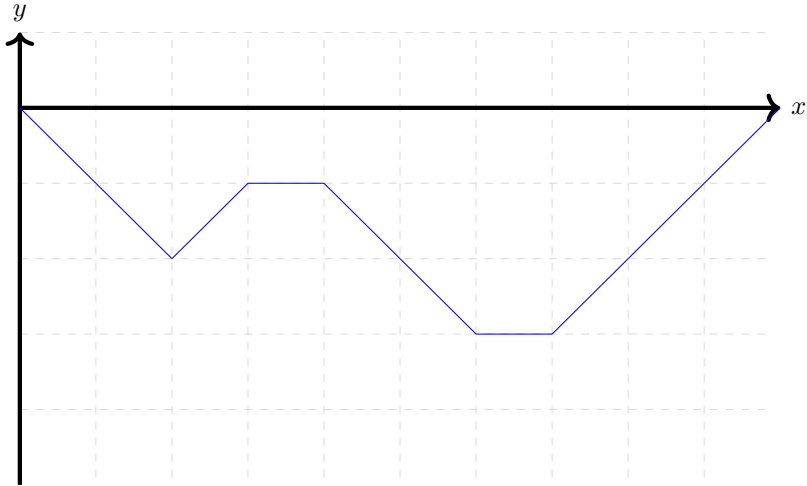

		So here each segment stands for a map of one of the three classes, noting that $\exp \in \TO$ and $\ln \in \FROM$, the amount of times one needs to precompose with $\exp$ is the negative height at the start of the segment and the amount of times one needs to postcompose with $\ln$ is the negative height at the end of the segment.\\
		
		It is then graphically obvious that the amount of $\exp$ and $\ln$ will cancel each other out and we are left with the same map, together with the remark that $\exp \in \TO$ and $\ln \in \FROM$ we have put each of these maps in some $A^{k}(\FROM\circ \mathcal{R}\circ \TO)$, $k$ being smaller than the largest negative height, which is the depth.
	\end{proof}
	
	The second Lemma is:
	
	\begin{Lemma}
		We have:
		
		\[\FROM\circ \mathcal{R}\circ \TO \subset \langle \Aff, \mathcal{R}^{0}, A\mathcal{R}^{0}, \mathcal{NC}\rangle\]
	\end{Lemma}
	
	\begin{proof}
		By definition $\FROM\circ \mathcal{R}\circ \TO$ is in:
		
		\[\Aff\circ \mathcal{NC}^{-1} \circ \mathcal{R}^{0} \circ \ln \circ \mathcal{R}^{0} \circ \mathcal{R} \circ \mathcal{R}^{0} \circ \exp \circ \mathcal{R}^{0} \circ \mathcal{NC}\circ \Aff\]
		
		which is in:
		
		\[\Aff\circ \mathcal{NC}^{-1} \circ \mathcal{R}^{0} \circ \ln \circ \mathcal{R} \circ \exp \circ \mathcal{R}^{0} \circ \mathcal{NC}\circ \Aff\]
		
		It is clear that it suffices to prove that:
		
		\[\ln \circ \mathcal{R}  \circ \exp  \subset \langle \Aff, \mathcal{R}^{0}, A\mathcal{R}^{0}, \mathcal{NC}\rangle\]
		
		Now clearly:
		
		\[\mathcal{R} \circ \exp \subset \mathcal{R}^{0} \circ \Aff \circ \exp\]
		
		Let us now take a closer look at $\Aff\circ \exp$, let $\alpha, \beta \in \mathbb{R}$ with $\alpha > 0$, then we are talking about a function of the form:
		
		\[f(\zeta) = \alpha e^{\zeta} + \beta = e^{\ln(\alpha e^{\zeta} + \beta)} = e^{\ln(\alpha) + \zeta + \ln\left(1 + \frac{\beta}{\alpha e^{\zeta}}\right)}\]
		
		And again the Taylor series of $\ln(1 + x)$ tell us that:
		
		\[\ln(\alpha) + z + \ln\left(1 + \frac{\beta}{\alpha e^{z}}\right) \in \mathcal{R}\]
		
		So:
		
		\[\Aff \circ \exp \subset \exp \circ \mathcal{R}^{0} \circ \Aff\]
		
		And thus:
		
		\[\ln \circ \mathcal{R} \circ \exp \subset \ln \circ \mathcal{R}^{0} \circ \exp \circ \mathcal{R}^{0} \circ \Aff \subset \langle A\mathcal{R}^{0}, \mathcal{R}^{0}, \Aff\rangle\]
		
		And we are done.
	\end{proof}
	
	We are now in a position to prove the Structural Theorem
	
	\begin{proof}[Proof of Structural Theorem]
		We already have that:
		
		\[G_{n}^{\log} \subset \langle A^{k}(\FROM\circ \mathcal{R}\circ \TO), \mathcal{R} \mid 0 \leq k \leq n - 1\rangle\]
		
		We will now prove by induction that:
		
		\[\langle A^{k}(\FROM\circ \mathcal{R}\circ \TO), \mathcal{R} \mid 0 \leq k \leq n - 1\rangle \subset \langle \Aff, A^{i}\mathcal{R}^{0}, A^{j}\mathcal{NC}\mid 0 \leq i \leq n, 0 \leq j \leq n - 1\rangle\]
		
		The case $n = 0$, i.e. $\mathcal{R} \subset \langle \Aff, \mathcal{R}^{0}\rangle$ is obvious. Before we do the induction step we claim that:
		
		\[A\Aff \subset \mathcal{R}\]
		
		And indeed, let $\alpha, \beta \in \mathbb{R}, \alpha > 0$, then:
		
		\[A(z \mapsto \alpha z + \beta) = z \mapsto \ln(\alpha e^{z} + \beta)\]
		
		and by assumption this is in $\mathcal{R}$.\\
		
		Assume by induction we have prove that:
		
		\[\langle A^{k}(\FROM\circ \mathcal{R}\circ \TO), \mathcal{R} \mid 0 \leq k \leq n - 2\rangle \subset \langle \Aff, A^{i}\mathcal{R}^{0}, A^{j}\mathcal{NC}\mid 0 \leq i \leq n - 1, 0 \leq j \leq n - 2\rangle\]
		
		We now want to prove that:
		
		\[\langle A^{k}(\FROM\circ \mathcal{R}\circ \TO), \mathcal{R} \mid 0 \leq k \leq n - 1\rangle \subset \langle \Aff, A^{i}\mathcal{R}^{0}, A^{j}\mathcal{NC}\mid 0 \leq i \leq n, 0 \leq j \leq n - 1\rangle\]
		
		It suffices by induction to prove that:
		
		\[A^{n - 1}(\FROM\circ \mathcal{R}\circ \TO) \subset \langle \Aff, A^{i}\mathcal{R}^{0}, A^{j}\mathcal{NC}\mid 0 \leq i \leq n, 0 \leq j \leq n - 1\rangle\]
		
		By the second Lemma we have:
		
		\[A^{n - 1}(\FROM\circ \mathcal{R}\circ \TO) \subset \langle A^{n - 1}\Aff, A^{n - 1}\mathcal{R}^{0}, A^{n}\mathcal{R}^{0}, A^{n - 1}\mathcal{NC}\rangle\]
		
		So it suffices to prove that:
		
		\[A^{n}\Aff \subset \langle \Aff, A^{i}\mathcal{R}^{0}, A^{j}\mathcal{NC}\mid 0 \leq i \leq n, 0 \leq j \leq n - 1\rangle\]
		
		Because we already know that $A \Aff \subset \mathcal{R}$, combined with the fact that $\exp \in \TO$ and $\ln \in \FROM$ we get for $n \geq 2$:
		
		\[A^{n}\Aff \subset A^{n - 2}(\FROM\circ \mathcal{R}\circ \TO)\]
		
		For $n = 1$ we simply have $\mathcal{R} \subset \Aff\circ \mathcal{R}^{0}$.\\
		
		So by induction we are done.
	\end{proof}
	\subsection{Fully stating goals}
	
	For this large inductive proof we will need to refer back to some intermediary lemma's, fixing some parameters in the lemma's. In order to do this smoothly we will add abbreviations after the relevant statements we will need to refer back to, together with the parameters in subscript. The order for the parameters is fixed, e.g. if in the original statement we put $A_{k, n, m}$ and later we write $A_{\alpha, \beta, \gamma}$ that means we refer to the statement $A_{k, n, m}$ where $k$ is replaced by $\alpha$, $n$ by $\beta$ and $m$ by $\gamma$.\\
	
	In order to deal with all the bookkeeping of \cite{ilyashenkoFiniteness}, we will first need some more definitions:
	
	\begin{Definition}
		On the germs of real analytic functions at infinity we will put the order given by saying that $f \leq g$ if for $x$ large enough:
		
		\[f(x) \leq g(x)\]
		
		Let $\alpha, \beta$ be other germs of functions at infinity, we say that $\alpha \preceq_{\Arch} \beta$ if there exist $C > 0$ such that for $x$ large enough:
		
		\[\left|\frac{\ln|\alpha(x)|}{\ln|\beta(x)|}\right| < C\]
		
		We say that $\alpha \sim_{\Arch}\beta$ if $\alpha \preceq_{\Arch}\beta$ and $\beta \preceq_{\Arch}\alpha$. We say for our germs of real analytic functions $f, g$ that $f\sim_{n} g$ if:
		
		\[e^{-\exp^{n}\circ f}\sim_{\Arch} e^{-\exp^{n}\circ g}\]
		
		If $G$ is some collection of germs of real analytic functions near infinity and $f \in G$, then we define:
		
		\[\overline{f}^{n, G} = \{g \in G \mid g \sim_{n} f\}\]
		
		We define analogously $f\preceq_{n} g$ and we define $f \prec_{n} g$ as $f \preceq_{n} g$ but not $f \sim_{n} g$.\\
		
		If it is clear from context, $n$ and/or $G$ may be dropped from notation.
	\end{Definition}
	
	\begin{Remark}
		Note that equivalently we could have defined the equivalence relation as follows: $f \sim_{0}g$ if and only if there exists some $0 < c < C < \infty$ such that for $x$ large enough:
		
		\[c < \frac{|f(x)|}{|g(x)|} < C\]
		
		And then $f\sim_{n + 1}g \Leftrightarrow e^{f} \sim_{n} e^{g}$.\\
		
		Note also that this is intimately related to the axioms by Remark \ref{RemA1Bounds}.
	\end{Remark}
	
	\begin{Example}
		We have for $c, d \in \mathbb{R}$: $x + c \sim_{0} x + d$, $x + c \sim_{1} x + d$ but unless $c = d$ we have $x + c \not\sim_{2} x + d$.
	\end{Example}
	
	Let us then write in full the Theorem we need to prove:
	
	\begin{Theorem}[Sufficiency of axioms (full)]\label{ThmSuffAxFull}
		For subsets $\mathcal{FC}^{n}_{i}$, $n \geq 0$, $i = 0, 1$, of germs of real analytic functions at infinity we define:
		
		\[\mathcal{F}^{n}_{ig} \coloneqq \mathcal{FC}_{i}^{n}\circ \exp^{n} \circ g \quad, \quad G^{-1} = \Aff \quad, \quad G^{0} = \mathcal{R}\]
		
		\[H_{0}^{n} \coloneqq \langle \id + \mathcal{F}_{0g}^{n} \mid g \in G^{n - 1}\rangle \quad, \quad H_{1}^{n} \coloneqq \langle \id + \mathcal{F}_{1g}^{n} \mid g \in G^{n - 1}\rangle\] 
		
		and we define inductively:
		
		\[J^{n - 1} \coloneqq \Ad(G^{n - 1})A^{n - 1}\mathcal{NC}\quad, \quad G^{n} \coloneqq \langle G^{n - 1}, J^{n - 1}, H^{n}_{0} \rangle\]
		
		Suppose there exists subsets $\mathcal{FC}_{i}^{n}$ satisfying the following axioms:
		
		\begin{enumerate}
			\item[$A1_{n, i}$]{For all $n$ and $i = 0, 1$, for any non-zero element $\alpha \in \mathcal{FC}^{n}_{i}$ there exist $\lambda < \mu < 0$ such that:
				
				\[e^{\lambda x} < \alpha < e^{\mu x}\]}
			\item[A2]{For all $n$:
				\begin{enumerate}
					\item[A2a]{$A^{n}\mathcal{R}^{0} \subset \id + \mathcal{F}^{n}_{0\id}$}
					\item[A2b]{$A^{n}\mathcal{NC} \subset \id + \mathcal{F}^{n}_{1\id}$}
			\end{enumerate}}
			\item[$A3_{n, i}$]{$\mathcal{FC}_{i}^{n}$ is a $\mathbb{R}$-vector space.}
			\item[A4]{For all $n \geq 0$, let $\rho$ be an element of any of the following classes of functions:
				
				\begin{enumerate}
					\item[$A4a_{n, i}$]{$\rho \in A^{-n}\overline{\id}^{n, G^{n - 1}}$.}
					\item[$A4b_{n, i}$]{$\rho \in \id + \left(\mathcal{FC}_{i}^{n}\circ A^{-n}(G^{n - 1}_{\preceq \id})\right)$ where $G^{n - 1}_{\preceq \id}$ are those elements of $G^{n - 1}$ which are $\preceq^{n} \id$.}
					\item[$A4c_{n, i}$]{If $n > 0$ or $i = 1$, $\rho \in A^{-n}J^{n - 1}$.}
				\end{enumerate}
				
				Then:
				
				\[\mathcal{FC}^{n}_{i}\circ \rho \subset \mathcal{FC}^{n}_{i}\]}
			\item[A5]{Let $k, n \geq 0, k < n$, let $i, j = 0, 1$. Let $\phi, \psi$ be part of any of the following classes of functions:
				
				\begin{enumerate}
					\item[$A5a_{k, n, i, j}$]{$\phi \in \mathcal{FC}_{i}^{k}$, $\psi \in \exp^{k}G^{n - 1}\ln^{n}$.}
					\item[$A5b_{n, i}$]{If $i = j$, $\phi \in \mathcal{FC}_{i}^{n}$, $\psi \in A^{-n}(G_{\prec\id}^{n - 1})$, where $G_{\prec\id}^{n - 1}$ is the part of $G^{n - 1}$ which is $\prec^{n} \id$ (i.e. $\preceq^{n} \id$ but not $\sim_{n} \id$).}
					\item[$A5c_{k, n, i}$]{$\phi = \exp$, $\psi \in \exp^{k}G^{n - 1}\ln^{n}$.}
				\end{enumerate}
				
				Then:
				
				\[\phi \circ (\psi + \mathcal{FC}_{j}^{n}) - \phi \circ \psi \subset \mathcal{FC}_{j}^{n}\]}
			\item[$A6_{n}$]{For all $n$ and $i = 0, 1$, $(\id  + \mathcal{FC}_{i}^{n})$ is closed under compositional inverses.}
			\item[$A7_{k, n, i}$]{For all $n \geq$, $0 \leq k \leq n$ and $i = 0, 1$, let $\cdot$ stand for pointwise multiplication, we have:
				
				\[e^{\mathcal{FC}^{n}_{i}} = 1 + \mathcal{FC}^{n}_{i}\]
				
				\[\left(\exp^{k}\circ G^{n - 1} \circ \ln^{n}\right)\cdot\mathcal{FC}^{n}_{i} = \mathcal{FC}^{n}_{i}\]}
			\item[$A8_{n}$]{We have $\mathcal{FC}_{0}^{n} \subset \mathcal{FC}_{1}^{n}$}
		\end{enumerate}
		
		Then all $(\mathcal{R}, \mathcal{NC})$-admissible polycycles have a (one-sided) neighbourhood without limit cycles.
	\end{Theorem}
	
	\begin{Remark}
		For notational convenience, i.e. in order to not have to split up formulas into cases all the time we will also introduce the notation that $J^{-1} = \{\id\}$.
	\end{Remark}
	
	The proof of this will be a large induction proof with many intermediate results. Let us first talk about the 'major theorems' in the sense that they tell us the intention and idea of the proof.\\
	
	Note that axiom $A2$ gets no indexing because we do not need it in the induction.
	
	\subsection{The major Theorems of the proof}\label{SubsecMajorThm}
	
	We intrinsically start from the Structural Theorem:
	
	\[G_{n}^{\log} \subset \langle \Aff, A^{i}\mathcal{R}^{0}, A^{j}\mathcal{NC}\mid 0 \leq i \leq n, 0 \leq j \leq n - 1\rangle\]
	
	From here we want to be able to rearrange the occurrences of $A^{i}\mathcal{R}^{0}$ and $A^{j}\mathcal{NC}^{0}$ in a convenient way in order to get a good normal form for elements of $G_{n}^{\log}$. For this we really need to enlarge the class of functions at which we are looking, such that we can rearrange these terms. This is the entire reason to introduce the complicated class of functional cochains $\mathcal{FC}^{n}_{i}, i = 0, 1, n \in \mathbb{N}$.\\
	
	Now what we want to study is $G^{n}$, that is where 'shuffling around the compositions' works. A first obvious theorem we need to have is:
	
	\begin{Theorem}[Inclusion Theorem, $IT_{n}$]
		We have:
		
		\[\langle \Aff, A^{i}\mathcal{R}^{0}, A^{j}\mathcal{NC}\mid 0 \leq i \leq n, 0 \leq j \leq n - 1\rangle \subset G^{n}\]
	\end{Theorem}
	
	As promised we need to be able to shift compositions around and this happens in the following form:
	
	\begin{Theorem}[Shift Theorem, $ST_{n}$]
		For all $n$ we have:
		
		\[J^{n - 1}G^{n - 1} \subset G^{n - 1}J^{n - 1}\]
		
		\[H^{n}_{0}J^{n - 1} \subset J^{n - 1}H^{n}_{0}\]
		
		\[H^{n}_{0}G^{n - 1} \subset G^{n - 1}H^{n}_{0}\]
		
		\[\mathcal{R}H_{0}^{0} \subset \Aff H_{0}^{0} \]
		
		In particular, remembering that $G^{0} = \mathcal{R}$:
		
		\[G^{n} = \Aff \circ H^{0}_{0}J^{0}H^{1}_{0}\cdots J^{n - 1}\circ H^{n}_{0}\]
		
		Furthermore, for all $n$ we have:
		
		\[H^{n}_{0}J^{n} \subset H^{n}_{1}\]
		
		So:
		
		\[G^{n} \subset \Aff \circ H^{0}_{1}H^{1}_{1}\cdots H^{n - 1}_{1}\circ H^{n}_{0}\]
	\end{Theorem}
	
	\begin{Remark}
		Note that $H^{n}_{1}$ is not defined to be in $G^{n}$, nonetheless we make the claim $ H^{n}_{0}J^{n} \subset H^{n}_{1}$.\\
		
		The idea is that $J^{n - 1}$ contains $A^{n - 1}\mathcal{NC} \subset H^{n - 1}_{1}$ by axiom A2b, then the formula follows from A8.\\
		
		In essence you would like to define $G^{n} = \langle H^{0}_{1}, ..., H^{n - 1}_{1}, H^{n}_{0}\rangle$, however then it becomes impossible for the domains to remain large enough for Phraghm\`en-Lindel\"of in the more general cases. The crucial difference lies in Proposition 3.3.12 point 2, this works for sectors (or exponentials of strips) but not for general domains of almost degree 1. The reason you get actual sectors (as exponentials of strips) is because you know that the domain of $\mathcal{NC}$ is a genuine strip but for elements of $FC^{n}_{1}$ you only get a domain of almost degree 0.\\
		
		In a sense this is making the correct arrangements for someone wishing to tackle a larger class of $(\mathcal{R}, \mathcal{NC})$ using these axiomatics.\\
	\end{Remark}
	
	Now before we do the final shifts we need a result on ordering
	
	\begin{Theorem}[Ordering Theorem, $OT_{n}$]
		The following hold:
		
		\begin{enumerate}
			\item{$G^{n - 1}$ is totally ordered by $\leq$, so in particular $G^{n - 1}/\sim_{n}$ is totally ordered by $\preceq$.}
			\item{If $f \sim_{n} g$, we have for $i = 0, 1: \mathcal{F}^{n}_{if} = \mathcal{F}^{n}_{ig}$, so sometimes we write $\mathcal{F}^{n}_{i\bar{f}}$.}
			\item{$\left(\id + \mathcal{F}^{n}_{if}\right)$ forms a group under composition and if $f \leq g$, then:
				
				\[\left(\id + \mathcal{F}^{n}_{ig}\right)\left(\id + \mathcal{F}^{n}_{if}\right) \subset \left(\id + \mathcal{F}^{n}_{if}\right)\left(\id + \mathcal{F}^{n}_{ig}\right)\]}
		\end{enumerate}
	\end{Theorem}
	
	From this we get:
	
	\begin{Theorem}[Multiplicative Decomposition Theorem, $MDT_{n}$]
		For each $n$ we can put a total order on couples $(m, \bar{g})$ where $0 \leq m \leq n$ and $\bar{g} \in G^{m}/\sim_{m}$ by saying that:
		
		\[(k, \bar{f}) < (m, \bar{g})\]
		
		if $k < m$ or $k = m$ and $\bar{f} < \bar{g}$.\\
		
		Let $g \in G^{n}$, then $g$ can be written as:
		
		\[g = a \circ (\id + \phi_{1}) \circ \cdots \circ (\id + \phi_{v}) \circ (\id + \psi_{1}) \circ \cdots \circ (\id + \psi_{w})\]
		
		where $a \in \Aff$, $\phi_{i} \in \mathcal{F}^{m_{i}}_{1\bar{f}_{i}}$, $\psi_{j} \in \mathcal{F}^{n}_{0\bar{g}_{j}}$ in such a way that:
		
		\[(m_{i}, \bar{f}_{i}) < (m_{i + 1}, \bar{f}_{i + 1})\]
		
		\[\bar{g}_{j} < \bar{g}_{j + 1}\]
	\end{Theorem}
	
	This then culminates in:
	
	\begin{Theorem}[Additive Decomposition Theorem, $ADT_{n}$]\label{ThmADT}
		Let $(k, \bar{f}) < (m, \bar{g})$, then for $i, j = 0, 1$:
		
		\[\mathcal{F}^{k}_{i\bar{f}}\circ (\id + \mathcal{F}^{m}_{j\bar{g}}) \subset \mathcal{F}^{k}_{i\bar{f}} + \mathcal{F}^{m}_{j\bar{g}}\]
		
		So for any $n$ any $g \in G^{n}$ can be written:
		
		\[g = a + \phi_{1} + \cdots + \phi_{v} + \psi_{1} + \cdots + \psi_{w}\]
		
		where $a \in \Aff$, $\phi_{i} \in \mathcal{F}^{m_{i}}_{1\bar{f}_{i}}$, $\psi_{j} \in \mathcal{F}^{n}_{0\bar{g}_{j}}$ in such a way that each $m_{i} < n$ and:
		
		\[(m_{i}, \bar{f}_{i}) < (m_{i + 1}, \bar{f}_{i + 1})\]
		
		\[\bar{g}_{j} < \bar{g}_{j + 1}\]
	\end{Theorem}
	
	From here the finiteness proof proceeds roughly by the method of Dulac, using the lower bound of axiom $A1$. I.e. a transfer map either has a non-trivial representation using ADT, in which case the lower bound from $A1$ gives a zero free region, or it has a trivial representation, in which case there are no isolated zeroes.\\
	
	\begin{Remark}
		Note that the structure of the ADT holds in particular formally. In particular it restricts the set of transseries which could be associated to a polycycle, for example
		
		\[\zeta + \sum_{k \geq 0} e^{-e^{k\zeta}}\]
		
		could never be associated to a polycycle because this would violate the finite sum of ADT at depth 1.
	\end{Remark}
	
	\subsection{Intermediate Lemma's}
	
	This subsection is a list of the technical lemma's used in proving sufficiency of the axioms. First come the Shift Lemma's
	
	\begin{Theorem}[Shift Lemma 1, $SL1_{n, i}$]
		For $i = 0, 1$, all $n$ and all $g \in \bar{\id} \subset G^{n}$:
		
		\[\mathcal{F}_{ig}^{n} = \mathcal{F}_{i\id}^{n}\]
	\end{Theorem}
	
	\begin{Theorem}[Shift Lemma 2, $SL2_{k, n, i, j}$]
		Let $(k - 1, \bar{f}) < (n - 1, \bar{g})$, then for $i, j = 0, 1$, $\Phi \in \mathcal{F}^{k}_{i\bar{f}}$ and $i = j$ if $k = n$:
		
		\[\Phi\circ (\id + \mathcal{F}^{n}_{j\bar{g}}) - \Phi \subset \mathcal{F}^{n}_{j\bar{g}}\]
	\end{Theorem}
	
	\begin{Theorem}[Shift Lemma 3]
		For all $n$:
		\begin{enumerate}
			\item[$SL3a_{n, i}$]{Let $f, g \in G^{n - 1}$, suppose that $\bar{g} \leq \bar{f}$. Then for $i = 0, 1$:
				
				\[\mathcal{F}_{if}^{n}\circ(\id + \mathcal{F}^{n}_{ig}) \subset \mathcal{F}_{if}^{n}\]}
			\item[$SL3b_{n, i}$]{We have for all $n$ and $i = 0, 1$:
				
				\[(\id + \mathcal{F}_{ig}^{n})^{-1} = \id + \mathcal{F}_{ig}^{n}\]}
		\end{enumerate}
	\end{Theorem}
	
	\begin{Theorem}[Shift Lemma 4]
		For all $n$, $i = 0, 1$:
		
		\begin{enumerate}
			\item[$SL4a_{n}$]{\[J^{n} \subset H_{1}^{n}\]}
			\item[$SL4b_{n, i}$]{\[\mathcal{F}_{ig}^{n}\circ J^{n - 1} \subset \mathcal{F}_{ig}^{n}\]}
		\end{enumerate}
	\end{Theorem}
	
	Next the conjugation Lemma's:
	
	\begin{Theorem}[Conjugation Lemma 1, $CL1_{n, i}$]
		Let $f, g \in G^{n - 1}$ be such that $\bar{f} < \bar{g}$, let $i = 0, 1$, then:
		
		\[\Ad\left(\id + \mathcal{F}_{if}^{n}\right)(\id + \mathcal{F}_{i\bar{g}}^{n}) \subset \id + \mathcal{F}_{ig}^{n}\]
	\end{Theorem}
	
	\begin{Theorem}[Conjugation Lemma 2, $CL2_{n}$]
		For all $n > 0$:
		
		\[\Ad(J^{n - 1})H^{n}_{0} = H^{n}_{0}\]
	\end{Theorem}
	
	\begin{Theorem}[Conjugation Lemma 3, $CL3_{n}$]
		For all $n$:
		
		\[\Ad(G^{n - 1})H_{0}^{n} = H^{n}_{0}\]
	\end{Theorem}
	
	\begin{Theorem}[Conjugation Lemma Regular, $CLR_{k, n, i}$]
		Let $k < n$, $f \in G^{k}$, $g \in G^{n - 1}$, let $i = 0, 1$, then:
		
		\[\Ad(f)(\id + \mathcal{F}^{n}_{ig}) \subset \id + \mathcal{F}^{n}_{ig\circ f}\]
	\end{Theorem}
	
	\subsection{The big induction}
	
	This entire subsection is dedicated to a proof of Sufficiency of the axioms together with the following supplement:
	
	\begin{Proposition}[Supplement to Sufficiency of axioms]\label{PropSupplAxioms}
		Suppose that $\mathcal{FC}_{i}^{n}$ are defined for all $n \leq N$ and $i = 0, 1$, satisfying all the axioms. Then:
		
		\begin{enumerate}
			\item{$ADT_{n}$ holds for all $n \leq N$.}
			\item{For all $n \leq N$, every $f \in G^{n}$ admits some $\lambda > \mu > 0$ such that for $x$ large enough:
				
				\[\lambda x \leq f(x) \leq \mu x\]
				
				Both bounds still hold even if only the upper bound of $A1_{n, 0}$ is known}
		\end{enumerate}
	\end{Proposition}
	
	So let us start:
	
	\begin{proof}[Proof of sufficiency of axioms and supplement.]
		Now, $IT_{n}$ clearly follows by induction from $A2a$ and then finiteness would follow from the usual Dulac argument using $ADT$, namely given a $\Delta \in G_{n}$, then it is either $\id$ or in the decomposition given by $ADT$, $\Delta - \id$ has a first non-zero term which dominates the others and prevents zeroes arbitrarily close to $0$ (as for why the first term in $ADT$ dominates, see point 14 of the induction).\\ 
		
		So let us talk about the structure of the induction needed to prove $ADT$. We represent the structure diagrammatically to give an overview of the inductive structure of this proof. The horizontal level indicates roughly how close results are to the axioms. Arrows in black go from above to below, arrows in green stay in the same level and arrows in red go back up the hierarchy, indicating induction. This diagram has been significantly simplified, so while this indicates the order in which things are proven, not all dependencies are written down explicitly, e.g. ADT directly uses SL2 even though there is no direct arrow between them.
		
		
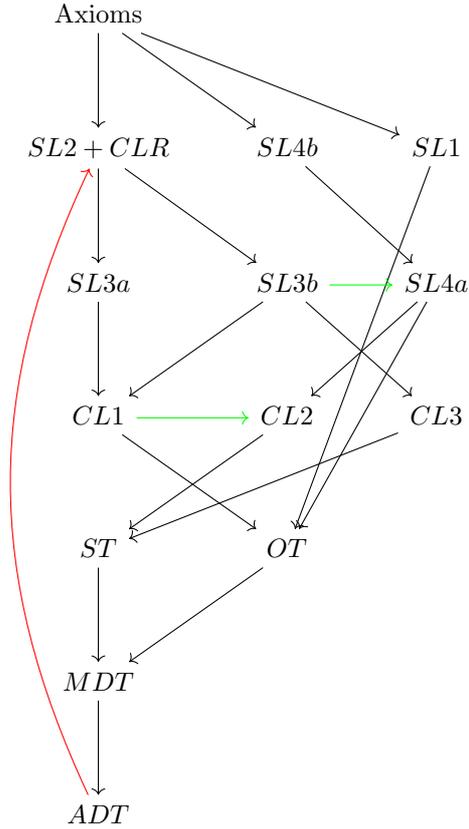
\[\hspace{-3.5cm}\begin{tikzcd}
			\text{Axioms}\arrow{dd}{}\arrow{ddr}{}\arrow{ddrr}{}&&\\&\\
			SL2 + CLR\arrow{dd}{}\arrow{ddr}{}& SL4b\arrow{ddr}{} & SL1\arrow{ddddddl}{}\\&\\
			SL3a\arrow{dd}{} & SL3b\arrow{ddl}{}\arrow[green]{r}{}\arrow{ddr}{} & \arrow{ddl}{}SL4a\arrow{ddddl}{} &\\&\\
			CL1\arrow{ddr}{}\arrow[green]{r}{} & CL2\arrow{ddl}{} & CL3\arrow{ddll}{}\\&\\
			ST\arrow{dd}{} & OT\arrow{ddl}{}\\&\\
			MDT\arrow{dd}{}\\&\\
			ADT\arrow[red, bend left = 25]{uuuuuuuuuu}{}
		\end{tikzcd}\]
		
		At this point it is clear that the induction, modulo the exact indices, works.\\
		
		The proof for the case $n = 0$ will not differ from the induction step much. So instead of proceeding as usual and first doing the base of induction, we will show the entire induction step and then at the end we will remark on what changes for the base case $n = 0$.\\
		
		\underline{\textbf{The induction step:}}\\
		
		Now we can go into each step of the induction, always assume that $n \geq 1$, $0 \leq k \leq n - 1$:
		
		\begin{enumerate}
			\item{\underline{$A4a_{n, i} \Rightarrow SL1_{n, i}$}\\
				
				We need to prove that for $g \in \bar{\id} \subset G^{n}$:
				
				\[\mathcal{FC}_{i}^{n}\circ \exp^{n}\circ g  \subset \mathcal{FC}_{i}^{n}\circ \exp^{n}\]
				
				Precomposing with $\ln^{n}$ on both sides we get equivalently:
				
				\[\mathcal{FC}_{i}^{n}\circ A^{-n}(g)\subset \mathcal{FC}_{i}^{n}\]
				
				which is exactly $A4a_{n, i}$.}
			\item{\underline{$CLR_{-1, n, i}$}\\
				
				Brute force calculation with $A3_{n}$.}
			\item{\underline{$CLR_{- 1, n, j} + A5a_{0, n, i, j} \Rightarrow SL2_{0, n, i, j}$}\\
				
				Let $f \in G^{- 1}$ and $g \in G^{n - 1}$. Then we need to prove that for $\phi \in \mathcal{FC}^{0}_{i}$:
				
				\[\phi\circ f\left(\id + \mathcal{FC}_{j}^{n}\circ\exp^{n}\circ g\right) - \phi\circ f \subset \mathcal{FC}_{j}^{n}\circ\exp^{n}\circ g\]
				
				Precomposing with $f^{-1}$ we need to prove that:
				
				\[\phi\left[\Ad(f^{-1})\left(\id + \mathcal{FC}_{j}^{n}\circ\exp^{n}\circ g\right)\right] - \phi \subset \mathcal{FC}_{j}^{n}\circ\exp^{n}\circ g\circ f^{-1}\]
				
				From $CLR_{- 1, n, j}$ we need to prove that:
				
				\[\phi\left(\id + \mathcal{FC}_{j}^{n}\circ\exp^{n}\circ g\circ f^{-1}\right) - \phi \subset \mathcal{FC}_{j}^{n}\circ\exp^{n}\circ g\circ f^{-1}\]
				
				or, precomposing with $f \circ g^{-1} \circ \ln^{n}$:
				
				\[\phi\left(f \circ g^{-1} \circ \ln^{n} + \mathcal{FC}_{j}^{n}\right) - \phi \circ f \circ g^{-1} \circ \ln^{n} \subset \mathcal{FC}_{j}^{n}\]
				
				This is just $A5a_{0, n, i, j}$.}
			\item{\underline{$CLR_{k - 1, n, j} + SL2_{k - 1, n, i, j} + A5a_{k, n, i, j} + A5b_{n, i} + A5c_{h, n, j}, h < k \Rightarrow SL2_{k, n, i, j}$}\\
				
				Let $f \in G^{k - 1}$ and $g \in G^{n - 1}$, satisfying $\bar{f} < \bar{g}$ if $k = n$. Then what we need to prove is that for $\phi \in \mathcal{FC}^{k}_{i}$:
				
				\[\phi\circ\exp^{k}\circ f\left(\id + \mathcal{FC}_{j}^{n}\circ\exp^{n}\circ g\right) - \phi\circ\exp^{k}\circ f \subset \mathcal{FC}_{j}^{n}\circ\exp^{n}\circ g\]
				
				Precomposing with $f^{-1}\circ \ln^{k}$ we need to prove that:
				
				\[\phi\left[A^{-k}\left(\Ad(f^{-1})\left(\id + \mathcal{FC}_{j}^{n}\circ\exp^{n}\circ g\right)\right)\right] - \phi \subset \mathcal{FC}_{j}^{n}\circ\exp^{n}\circ g\circ f^{-1}\circ \ln^{k}\]
				
				From $CLR_{k - 1, n, j}$ we need to prove that:
				
				\[\phi\left[\exp^{k}\circ \left(\id + \mathcal{FC}_{j}^{n}\circ\exp^{n}\circ g\circ f^{-1}\right)\circ \ln^{k}\right] - \phi \subset \mathcal{FC}_{j}^{n}\circ\exp^{n}\circ g\circ f^{-1}\circ \ln^{k}\]
				
				Now we want to prove that:
				
				\[\exp^{k}\circ \left(\id + \mathcal{FC}_{j}^{n}\circ\exp^{n}\circ g\circ f^{-1}\right) - \exp^{k}\subset  \mathcal{FC}_{j}^{n}\circ\exp^{n}\circ g\circ f^{-1}\]
				
				We do this by induction on $k$, so suppose this proven for $k = h$, then we consider the case $k = h + 1$
				
				\[\exp^{h + 1}\circ \left(\id + \mathcal{FC}_{j}^{n}\circ\exp^{n}\circ g\circ f^{-1}\right) - \exp^{h + 1}\subset  \mathcal{FC}_{j}^{n}\circ\exp^{n}\circ g\circ f^{-1}\]
				
				by induction it suffices to prove that:
				
				\[\exp\circ \left(\exp^{h} + \mathcal{FC}_{j}^{n}\circ\exp^{n}\circ g\circ f^{-1}\right) - \exp^{h + 1}\subset  \mathcal{FC}_{j}^{n}\circ\exp^{n}\circ g\circ f^{-1}\]
				
				precomposing with $f \circ g^{-1} \circ \ln^{n}$ we get that we need to prove that:
				
				\[\exp\circ \left(\exp^{h}\circ f \circ g^{-1} \circ \ln^{n} + \mathcal{FC}_{j}^{n}\right) - \exp \circ \exp^{h} \circ f \circ g^{-1} \circ \ln^{n}\subset  \mathcal{FC}_{j}^{n}\]
				
				Note that this is just $A5c_{h, n, j}$ and we are done with this part. From this we get that:
				
				\[\exp^{k}\circ \left(\id + \mathcal{FC}_{j}^{n}\circ\exp^{n}\circ g\circ f^{-1}\right)\circ \ln^{k}\subset \id +  \mathcal{FC}_{j}^{n}\circ\exp^{n}\circ g\circ f^{-1} \circ \ln^{k}\]
				
				So we only need to prove that:
				
				\[\phi\left[\id + \mathcal{FC}_{j}^{n}\circ\exp^{n}\circ g\circ f^{-1}\circ \ln^{k}\right] - \phi \subset \mathcal{FC}_{j}^{n}\circ\exp^{n}\circ g\circ f^{-1}\circ \ln^{k}\]
				
				We can then precompose by $\exp^{k}\circ f \circ g^{-1} \circ \ln^{n}$ to get:
				
				\[\phi\left[\exp^{k}\circ f \circ g^{-1} \circ \ln^{n} + \mathcal{FC}_{j}^{n}\right] - \phi\circ \exp^{k}\circ f \circ g^{-1} \circ \ln^{n} \subset \mathcal{FC}_{j}^{n}\]
				
				Then depending on whether $k = n - 1$ this is either $A5a_{k, n, i, j}$ or $A5b_{n, i}$, in the second case an elementary calculation shows that $fg^{-1} <\id$.}
			\item{\underline{$SL2_{h, n, 1, i}, h < k - 1 + SL2_{k - 1, n, 0, i}  + ADT_{k} \Rightarrow CLR_{k, n, i}, k \leq n - 1$}\\
				
				We need to prove that for $f \in G^{k}$ and $g\in G^{n - 1}$:
				
				\[f^{-1}\circ(\id + \mathcal{F}^{n}_{ig})\circ f \subset \id + \mathcal{F}^{m}_{igf}\]
				
				Using $ADT_{k}$ we can decompose:
				
				\[f^{-1} = a + \sum \phi_{j}\]
				
				$a$ being the affine part. Now by $A3_{n, i}$:
				
				\[a \circ (\id + \mathcal{F}^{n}_{ig}) \subset a + \mathcal{F}^{n}_{ig}\]
				
				Then by the appropriate $SL2_{h, n, j, i}$ for all $h \leq k - 1$ together with $A3_{n, i}$:
				
				\[\left(\sum \phi_{j}\right)\circ (\id + \mathcal{F}^{n}_{ig}) \subset \sum \phi_{j} + \mathcal{F}^{n}_{ig}\]
				
				Combining the two we get:
				
				\[f^{-1}\circ(\id + \mathcal{F}^{n}_{ig})\circ f \subset (f^{-1} + \mathcal{F}^{n}_{ig}) \circ f = \id + \mathcal{F}^{m}_{igf}\]}
			\item{\underline{$CLR_{n - 1, n, i} + A4b_{n, i} + A5c_{h, n, i}, h < n \Rightarrow SL3a_{n, i}$}\\
				
				We need to prove that for $f, g \in G^{n - 1}, \bar{g} \leq \bar{f}$:
				
				\[\mathcal{FC}_{i}^{n}\circ\exp^{n}\circ f (\id + \mathcal{FC}_{i}^{n}\circ\exp^{n}\circ g) \subset \mathcal{FC}_{i}^{n}\circ\exp^{n}\circ f\]
				
				Or upon precomposing by $f^{-1}\circ \ln^{n}$ we get:
				
				\[\mathcal{FC}_{i}^{n}\circ \Ad(\ln^{n})\left[\Ad(f^{-1})(\id + \mathcal{FC}_{i}^{n}\circ\exp^{n}\circ g)\right] \subset \mathcal{FC}_{i}^{n}\]
				
				By $CLR_{n - 1, n, i}$ this means proving that:
				
				\[\mathcal{FC}_{i}^{n}\circ \Ad(\ln^{n})(\id + \mathcal{FC}_{i}^{n}\circ\exp^{n}\circ g\circ f^{-1}) \subset \mathcal{FC}_{i}^{n}\]
				
				Using $A5c_{h, n, i}$ as in point $4$ we see that we only need to prove that:
				
				\[\mathcal{FC}_{i}^{n}\circ (\id + \mathcal{FC}_{i}^{n}\circ\exp^{n}\circ g\circ f^{-1}\circ \ln^{n}) \subset \mathcal{FC}_{i}^{n}\]
				
				This is exactly $A4b_{n, i}$}
			\item{\underline{$CLR_{n - 1, n, i} + A7_{n, i} + A6_{n, i} \Rightarrow SL3b_{n, i}$}\\
				
				So we need to prove that for $g \in G^{n - 1}$:
				
				\[(\id + \mathcal{F}_{ig}^{n})^{-1} \subset \id + \mathcal{F}_{ig}^{n}\]
				
				By $CLR_{n - 1, n, i}$:
				
				\[\Ad(g)(\id + \mathcal{F}_{i\id}^{n}) \subset (\id + \mathcal{F}_{ig}^{n})\]
				
				It is not hard to see that this is equivalent to:
				
				\[(\id + \mathcal{F}_{i\id}^{n}) \subset \Ad(g^{-1})(\id + \mathcal{F}_{ig}^{n})\]
				
				So by $CLR_{n - 1, n, i}$ again, we know that
				
				\[\Ad(g)(\id + \mathcal{F}_{i\id}^{n}) = (\id + \mathcal{F}_{ig}^{n})\]
				
				We now claim that for $k \leq n$:
				
				\[A^{-k}(\id + \mathcal{FC}_{i}^{n}\exp^{n}) = (\id + \mathcal{FC}_{i}^{n}\exp^{n - k})\]
				
				we prove this by induction on $k$, it is clear that all we need to prove is that for all $k < n$:
				
				\[\exp \circ (\id + \mathcal{FC}_{i}^{n}\exp^{n - k}) \circ \ln = (\id + \mathcal{FC}_{i}^{n}\exp^{n - k - 1})\]
				
				And:
				
				\[\exp \circ (\id + \mathcal{FC}_{i}^{n}\exp^{n - k}) \circ \ln = \exp \circ (\ln + \mathcal{FC}_{i}^{n}\exp^{n - k - 1})\]
				
				remembering that $\cdot$ stands for pointwise multiplication, then by axiom $A7_{n, i}$:
				
				\[\exp \circ (\id + \mathcal{FC}_{i}^{n}\exp^{n - k}) \circ \ln = \id \cdot\left[\exp \circ (\mathcal{FC}_{i}^{n}\exp^{n - k - 1})\right]\]
				
				which is equal to:
				
				\[\id\cdot(1 + \mathcal{FC}_{i}^{n}\exp^{n - k - 1}) = \id + \mathcal{FC}_{i}^{n}\exp^{n - k - 1}\]
				
				So we are done with this part. In particular we can conclude that:
				
				\[A^{-n}(\id + \mathcal{FC}_{i}^{n}\exp^{n}) = (\id + \mathcal{FC}_{i}^{n})\]
				
				So it suffices to show that $(\id + \mathcal{FC}_{i}^{n})$ is closed under compositional inverses, which is exactly $A6_{n, i}$.}
			\item{\underline{$A4c_{n, i} \Rightarrow SL4b_{n, i}$}\\
				
				Let $i = 0, 1$. We need to prove that for $g \in G^{n - 1}$:
				
				\[\mathcal{FC}_{i}^{n}\circ \exp^{n} \circ g \circ J^{n - 1} \subset \mathcal{FC}_{i}^{n}\circ \exp^{n} \circ g\]
				
				precomposing with $g^{-1} \circ \ln^{n}$ we get:
				
				\[\mathcal{FC}_{i}^{n} \circ A^{-n}(\Ad(g^{-1})(J^{n - 1})) \subset \mathcal{FC}_{i}^{n}\]
				
				Now by definition:
				
				\[J^{n - 1} = \Ad(G^{n - 1})A^{n - 1}\mathcal{NC}\]
				
				So $\Ad(g^{-1})(J^{n - 1}) \subset J^{n - 1}$, so we are only left with:
				
				\[\mathcal{FC}_{i}^{n} \circ A^{-n}(J^{n - 1}) \subset \mathcal{FC}_{i}^{n}\]
				
				Which is exactly $A4c_{n, i}$.}
			\item{\ul{$ST_{n} + CLR_{n - 1, n, 1} + SL2_{n - 1, n, 1, 1} + A2b_{n} + SL4a_{n - 1} + A8_{n} + SL4b_{n, 1} + CL1_{n - 1, 1}$ $ + SL2_{n - 1, n - 1, 1, 1} \Rightarrow SL4a_{n}$}\\
				
				So we need to prove that:
				
				\[J^{n} = \Ad(G^{n})A^{n}\mathcal{NC} \subset H^{n}_{1}\]
				
				By $A2b_{n}$ we know that $A^{n}\mathcal{NC} \subset H^{n}_{1}$ so we only need to prove:
				
				\[\Ad(G^{n})H^{n}_{1} \subset H^{n}_{1}\]
				
				By $ST_{n}$ we have $G^{n} = G^{n - 1}J^{n - 1}H^{n}_{0}$. So we only need to prove:
				
				\[\Ad(H^{n}_{0})\left[\Ad(J^{n - 1})\left[\Ad(G^{n - 1})(H^{n}_{1})\right]\right] \subset H^{n}_{1}\]
				
				Because conjugation commutes with composition it suffices to prove that for a $g \in G^{n - 1}$:
				
				\[\Ad(H^{n}_{0})\left[\Ad(J^{n - 1})\left[\Ad(G^{n - 1})(\id + \mathcal{F}^{n}_{1g})\right]\right] \subset H^{n}_{1}\]
				
				By $CLR_{n - 1, n, 1}$ conjugating by an element in $G^{n - 1}$ can only change the $g$ below, so it suffices to prove that:
				
				\[\Ad(H^{n}_{0})\left[\Ad(J^{n - 1})(\id + \mathcal{F}^{n}_{1g})\right] \subset H^{n}_{1}\]
				
				By $SL4a_{n - 1}$ we have $J^{n - 1} \subset H^{n - 1}_{1}$. This means that a $j \in J^{n - 1}$ (or $j^{-1}$ as well) can be written as:
				
				\[j \in (\id + \mathcal{F}_{1g_{1}}^{n - 1})\circ \cdots \circ (\id + \mathcal{F}_{1g_{h}}^{n - 1})\]
				
				Using $CL1_{n - 1, 1}$ to reorder if necessary we can use $SL2_{n - 1, n - 1, 1, 1}$ to write $j$ as:
				
				\[j \in \id + \sum_{\ell = 1}^{h}\mathcal{F}_{1g_{\ell}}^{n - 1}\]
				
				Arguing using $SL2_{n - 1, n, 1, 1}$ we can get as in point 5:
				
				\[\Ad(j)(\id + \mathcal{F}^{n}_{1g}) \subset \id + \mathcal{F}^{n}_{1g}\circ j\]
				
				So from $SL4b_{n, 1}$ we may conclude:
				
				\[\Ad(J^{n - 1})(\id + \mathcal{F}^{n}_{1g}) \subset \id + \mathcal{F}^{n}_{1g}\]
				
				So all we need to prove still is:
				
				\[\Ad(H^{n}_{0})(\id + \mathcal{F}^{n}_{1g}) \subset H^{n}_{1}\]
				
				But by axiom $A8_{n}$, $H^{n}_{0} \subset H^{n}_{1}$ and we are done.}
			\item{\underline{$SL2_{n - 1, n, i, i} + SL3a_{n, i} + SL3b_{n, i} \Rightarrow CL1_{n, i}$}\\
				
				Let $f, g \in G^{n - 1}$, $\bar{f} < \bar{g}$, $i = 0, 1$, let $\phi, \tilde{\phi} \in \mathcal{F}^{n}_{if}$ using $SL3b_{n, i}$ be such that:
				
				\[(\id + \phi)^{-1} = \id + \tilde{\phi}\]
				
				Let $\psi \in \mathcal{F}^{n}_{ig}$ we need to prove that:
				
				\[(\id + \tilde{\phi})(\id + \psi)(\id + \phi) \in \id + \mathcal{F}_{ig}^{n}\]
				
				Now by $SL2_{n - 1, n, i, i}$ and $A3_{n, i}$:
				
				\[(\id + \tilde{\phi})(\id + \psi) \in \id + \psi + \tilde{\phi} + \mathcal{F}^{n}_{ig} \subset (\id + \tilde{\phi}) + \mathcal{F}^{n}_{ig}\]
				
				Then using $SL3a_{n, i}$:
				
				\[\left[(\id + \tilde{\phi}) + \mathcal{F}^{n}_{ig}\right](\id + \phi) \subset \id + \mathcal{F}^{n}_{ig}\]}
			\item{\underline{$SL2_{n - 1, n, 1, 0} + SL2_{n - 1, n - 1, 1, 1} + SL3b_{n, 0} + SL4b_{n, 0} + SL4a_{n - 1} + CL1_{n - 1, 1} \Rightarrow CL2_{n}$}\\
				
				So we need to prove that for all $f \in G^{n - 2}$:
				
				\[\Ad(J^{n - 1})H^{n}_{0} \subset H^{n}_{0}\]
				
				Because conjugation commutes with composition combined with $SL3b_{n, 0}$ it suffices to consider for all $g \in G^{n - 1}$:
				
				\[\Ad(J^{n - 1})(\id + \mathcal{F}^{n}_{0g}) \subset H^{n}_{0}\]
				
				As in point 9, we have that:
				
				\[\Ad(J^{n - 1})(\id + \mathcal{F}^{n}_{0g}) \subset \id + \mathcal{F}^{n}_{0g}\circ J^{n - 1}\]
				
				Using $SL4b_{n, 0}$ we are done.}
			\item{\underline{$CLR_{n - 1, n, 0} + SL3b_{n, 0} \Rightarrow CL3_{n}$}\\
				
				We need to prove that:
				
				\[\Ad(G^{n - 1})(H^{n}_{0}) \subset H^{n}_{0}\]
				
				Conjugation commutes with composition combined with $SL3b_{n}$ says that it suffices to show that for all $g \in G^{n - 1}$:
				
				\[\Ad(G^{n - 1})(\id + \mathcal{F}^{n}_{0g}) \subset H^{n}_{0}\]
				
				And this follows from $CLR_{n - 1, n}$.}
			\item{\underline{$CL2_{n} + CL3_{n} + A8_{h} + SL4a_{h} + ST_{h}, h < n \Rightarrow ST_{n}$}\\
				
				Let us go across the statements one by one:
				
				\begin{enumerate}
					\item{By definition:
						
						\[J^{n - 1} = \Ad(G^{n - 1})A^{n - 1}\mathcal{NC}\]
						
						So obviously for $j \in J^{n - 1}$ and $f \in G^{n - 1}$:
						
						\[jf = f\Ad(f)(j) \in G^{n - 1}J^{n - 1}\]}
					\item{That $H^{n}_{0}J^{n - 1} \subset J^{n - 1}H^{n}_{0}$ follows easily from $CL2_{n}$.}
					\item{That $H^{n}_{0}G^{n - 1} \subset G^{n - 1}H^{n}_{0}$ follows easily from $CL3_{n}$}
					\item{$\mathcal{R}H_{1}^{0} \subset \mathcal{A}ffH_{1}^{0}$ obviously follows from $A2a_{0}$ and $A8_{0}$ as follows:
						
						\[\mathcal{R}H_{1}^{0} \subset \Aff \mathcal{R}^{0}H_{1}^{0} \subset \Aff H_{0}^{0}H_{1}^{0} \subset \Aff H_{1}^{0}\].}
					\item{The first form of $G^{n}$ follows easily from induction using the $ST_{h}$.}
					\item{We need to prove that $H_{0}^{h}J^{h} \subset H^{h}_{1}$ for $h < n$ to get to the final form of elements in $G^{n}$, by $A8_{h}$, $H_{0}^{h} \subset H^{h}_{1}$ and by $SL4a_{h}$ we get that $J^{h} \subset H^{h}_{1}$, so we are done.}
			\end{enumerate}}
			\item{\ul{$A1_{h, 0} + A1_{h, 1} + ADT_{h} + SL4a_{h}, h < n - 1 + A1_{n - 1, 0} + ADT_{n - 1} + CL1_{n}$ $ + SL1_{n} + SL3a_{n}$ $ + SL3b_{n} \Rightarrow OT_{n}$}\\
				
				Let us go over the statements one by one:
				
				\begin{enumerate}
					\item{It suffices to prove that $G^{n - 1}$ is totally ordered, so it suffices to prove that the difference of two distinct elements in $G^{n - 1}$ is eventually either positive or negative. By $ADT_{n - 1}$ it suffices to prove that any 
						
						\[g = a + \phi_{1} + \cdots + \phi_{v} + \psi_{1} + \cdots + \psi_{w}\]
						
						where $a \in \mathcal{A}ff$, $\phi_{i} \in \mathcal{F}^{m_{i}}_{1\bar{f}_{j}}$, $\psi_{j} \in \mathcal{F}^{n - 1}_{0\bar{g}_{j}}$ in such a way that each $m_{i} < n - 1$ and:
						
						\[(m_{i}, \bar{f}_{i}) < (m_{i + 1}, \bar{f}_{i + 1})\]
						
						\[\bar{g}_{j} < \bar{g}_{j + 1}\]
						
						is eventually either positive or negative. For this we need to consider multiple statements on growth:
						
						\begin{enumerate}
							\item{By $A1_{k}$, for each $k < n$ and each $\alpha \in \mathcal{FC}_{i}^{k}$ there exists $\lambda < \mu < 0$ such that in the same ordering as on $G^{n - 1}$:
								
								\[e^{\lambda x} < \alpha < e^{\mu x}\]}
							\item{By strong induction for every $k \leq n - 1$, any $f \in G^{k}$ grows to plus infinity. This is clear for $k = -1$, for $k = 0$ this is true by admissibility. Now by definition:
								
								\[G^{k} = \langle G^{k - 1}, J^{k - 1}, H^{k}_{0}\rangle\]
								
								the composition of two functions going to $+\infty$ goes to $+\infty$ the compositional inverse of a function going to $+\infty$ goes to $+\infty$, so by induction it suffices to show that $J^{k - 1}$ and $H^{k}_{0}$ contains only functions going to $+\infty$.\\
								
								By definition of $H^{k}_{0}$ and induction every function in $H^{k}_{0}$ is dominated by the identity part so goes to $+\infty$.\\
								
								By $SL4a_{k - 1}$ we have $J^{k - 1} \subset H^{k - 1}_{1}$ and by induction assumption any function in $H^{k - 1}_{1}$ is dominated by the identity part and goes to $+\infty$.}
							\item{By strong induction for every $k \leq n - 1$, any $f \in G^{k}$ grows at most linearly in absolute value. Indeed, this is clearly true for $G^{-1}$, suppose this proven up to $k - 1$, by $ADT_{k}$ combined with the previous growth estimates and the strong induction hypothesis one sees easily that any element in $G^{k}$ grows at most linearly.}
							\item{The final statement is that the leading term is dominant, in the sense that all other terms are small-$o$ of the leading term.\\
								
								If after the leading term the $m_{i}$ increases this is obvious by the previous estimates. So suppose $m_{i}$ does not immediately increase. Note that we can never be talking about a leading affine term.\\ 
								
								So let us fix an $m \leq n$ and let us take $f, g \in G^{m - 1}$ with $\bar{f} < \bar{g}$ what we want to prove is that for any $\alpha, \beta \in \mathcal{FC}^{m}_{i}$:
								
								\[\lim_{x \to \infty}\frac{\alpha(\exp^{m}(f(x)))}{\beta(\exp^{m}(g(x)))} = 0\]
								
								or alternatively by the first growth estimate we want for all $\lambda, \mu < 0$:
								
								\[\lim_{x \to \infty}\frac{e^{\lambda\exp^{m}(f(x))}}{e^{\mu\exp^{m}(g(x))}} = 0\]
								
								But note that this follows exactly from $\bar{f} \neq \bar{g}$, by definition of the relation $\sim_{m}$ together with the fact that $\bar{f} < \bar{g}$.} 
						\end{enumerate}
						
						So we know the leading term in the difference of two elements in $G^{n - 1}$ is dominant enough to determine whether the function eventually becomes positive or negative, noting that the decomposition given by $ADT$ is finite. So indeed $G^{n - 1}$ is totally ordered.}
					\item{What we need to prove is that if $f \sim_{n} g$, then:
						
						\[\mathcal{FC}^{n}_{i}\exp^{n}\circ f = \mathcal{FC}^{n}_{i}\exp^{n}\circ g\]
						
						So it suffices to prove that:
						
						\[\mathcal{FC}^{n}_{i}\exp^{n}\circ f\circ g^{-1} = \mathcal{FC}^{n}_{i}\exp^{n}\]
						
						It is easy to prove that then $f\circ g^{-1} \sim_{n} \id$ and thus this is $SL1_{n}$.}
					\item{The relation that if $f, g \in G^{n - 1}$ and $\bar{f} < \bar{g}$ then:
						
						\[(\id + \mathcal{F}^{n}_{i\bar{g}})(\id + \mathcal{F}^{n}_{i\bar{f}}) \subset (\id + \mathcal{F}^{n}_{i\bar{f}})(\id + \mathcal{F}^{n}_{i\bar{g}})\]
						
						Follows easily from $CL1_{n}$. So we only need to focus on the assertion that $(\id + \mathcal{F}^{n}_{i\bar{f}})$ is a group under composition. That it is closed under inverses is $SL3b_{n}$. So we only need to consider compositions, and by $SL3a_{n}$ and $A3_{n}$:
						
						\[(\id + \mathcal{F}^{n}_{i\bar{f}})(\id + \mathcal{F}^{n}_{i\bar{f}}) \subset \id + \mathcal{F}^{n}_{i\bar{f}} + \mathcal{F}^{n}_{i\bar{f}} \subset \id + \mathcal{F}^{n}_{i\bar{f}}\]}
			\end{enumerate}}
			\item{\underline{$ST_{n} + OT_{n} \Rightarrow MDT_{n}$}\\
				
				Trivial.}
			\item{\underline{$MDT_{n} + SL2_{v, h, 1, 1} + SL2_{v, n, 1, 0}, v \leq h < n,   \Rightarrow ADT_{n}$}\\
				
				We work by induction on the amount of factors in the $MDT$. The case of having a single factor is obvious. So suppose we have proven this up to having $m - 1$ factors and we want to prove the case of $m$ factors. Let $f \in G^{n}$ have $m$ factors under $MDT_{n}$, the final factor being in $(\id + \mathcal{F}^{j}_{ig})$. By induction hypothesis we may rewrite all previous factors as $a + \sum \phi_{l}$ with $\phi_{l} \in \mathcal{F}^{k_{l}}_{ih_{l}}$.\\
				
				By $A3_{n}$:
				
				\[a(\id + \mathcal{F}^{j}_{ig}) \subset a + \mathcal{F}^{j}_{ig}\]
				
				And for each $l$, by $SL2_{j, k_{l}}$, by the ordering inherent in $MDT_{n}$ this makes sense, we have:
				
				\[\phi(\id + \mathcal{F}^{j}_{ig}) \subset \phi + \mathcal{F}^{j}_{ig}\]
				
				So by $A3_{n}$:
				
				\[f \in a + \sum \phi_{l} + \mathcal{F}^{j}_{ig}\]
				
				which wraps up the induction.}
		\end{enumerate}
		
		\underline{\textbf{Base of induction:}}\\
		
		Now all we are missing is the base of induction, i.e. the case $n = 0$, let us go over how all these go:\\
		
		First the ones that go exactly as in the inductive case: $SL1, SL2$ (which is now only $SL2_{0, 0, i, j}$ so does not need $ADT$), $CLR$ (same remark as $SL2$), $SL3a$, $SL3b$, $CL1$, $CL3$, $ST$ (assuming $SL4a$), $MDT$ (assuming $SL4a$), $ADT$ (assuming $SL4a$).\\ 
		
		It is worth noting that $CL2_{0}$ will never be needed because in Shift Theorem if we arrive at $G^{0}$, we know that this is just $\mathcal{R}$ and then we use the fact that $\mathcal{R}H^{0}_{0} \subset \Aff H^{0}_{0}$. Because after this we give a proof of $SL4a_{0}$ without using $SL4b_{0}$ we also do not need $SL4b_{0}$.\\
		
		Then there are only two left:
		
		\begin{enumerate}
			\item{$\underline{SL4a_{0}}$\\
				
				This is the same as the induction, the problem is just that we have not explicitly defined $J^{-1}$. Literally written it says that:
				
				\[J^{0} = \Ad(G^{0})\mathcal{NC} = \Ad(\mathcal{R})\mathcal{NC} \subset H^{0}_{1}\]
				
				and $\mathcal{R} \subset \langle \Aff, \mathcal{R}^{0}\rangle$, so using axiom $A2b_{0}$ it suffices to consider elements of the form:
				
				\[\Ad(g)\left(\id + \mathcal{FC}^{0}_{1} \circ h\right)\]
				
				where $h \in \Aff$ and either $g \in \Aff$ or $g \in \mathcal{R}^{0}$, the case where $g \in \Aff$ is just $CLR$. So we only need to consider the case where $g \in \mathcal{R}^{0}$. But by axiom $A2a_{0}$, $\mathcal{R}^{0} \subset \id + \mathcal{FC}_{0}^{0}$, so by axiom $A8_{0}$ we are done.}
			\item{\underline{$OT_{0}$}\\
				
				Statements $2$ and $3$ are exactly the same as in the induction. Statement $1$ follows in fact from admissibility, indeed suppose that $f, g \in \mathcal{R}$ are not ordered, i.e. there is an infinitely increasing sequence of $x_{i}$ such that:
				
				\[f(x_{i}) = g(x_{i})\]
				
				then there is an infinitely increasing sequence $y_{i} \coloneqq g(x_{i})$ (remember the linear part in $\Aff$ is always positive) such that:
				
				\[f(g^{-1}(y_{i})) = y_{i}\]
				
				so we can look at $f \circ g^{-1}$ which is equal to the identity at an infinity of points, so by admissibility we have $f \circ g^{-1} = \id$.}
		\end{enumerate}
		
		\underline{\textbf{Supplement:}}\\
		
		Everything is proven almost literally, point 1 of the supplement follows from the nature of the induction and point 2 of the supplement is done in point 14 (a) iii of the induction step (note that by taking inverses the upper bounds give a lower bound).
	\end{proof}
	
	\section{Domains for the convergent case}
	
	\begin{Remark}
		All arguments in this section are really modulo any compact sets. We are interested in behaviour near infinity of domains. For those more formally inclined, it is possible to put on all subsets of $\mathbb{C}$ an equivalence relation by saying that domains $A$ and $B$ are equivalent if their symmetric difference, i.e. $(A\setminus B) \cup (B \setminus A)$ is precompact (i.e. it is bounded).\\
		
		If we say some statement is true for a domain $A$ what we formally mean is that there exists some domain $B$ equivalent to $A$ for which this statement holds. The reason this works is in some sense that Phraghm\`en-Lindel\"of does not care about compact sets.\\
		
		We will also assume that any domain we speak of has points with arbitrarily large positive real part.
	\end{Remark}
	
	\subsection{Arguments for the convergent case}
	
	It is clear that in order to prove Theorem \ref{ThmConvPoly} all we need to do is to find such sets of $\mathcal{FC}_{i}^{n}$. Here we have a few remarks:
	
	\begin{enumerate}
		\item{If we have multiple (finite of infinite) sets of possible functional cochains, the intersection of all of them is still a valid candidate, in other words there exists a minimal set of functional cochains.}
		\item{Suppose there exists a valid set of functional cochains. Suppose given a set of functions which are nearly functional cochains in the sense that they satisfy axioms $A2-A8$, but not $A1$, then by intersecting with a valid set of functional cochains we get a subset which does satisfy all axioms for functional cochains.}
	\end{enumerate}
	
	Combining these two remarks we get the following: (Given an admissible choice of $(\mathcal{R}, \mathcal{NC})$) Suppose there exist a valid set of functional cochains, then the smallest set of functions satisfying axioms $A2-A8$, must also satisfy axiom $A1$.\\
	
	And it is actually easy to construct the smallest set satisfying axioms $A2-A8$.
	
	\begin{Definition}
		Let us define on induction on $k$ and $n$ $FC^{n, k, 0}$ and $FC^{n, k, 1}$, the case $k = 0$ will be treated separately but not the case $n = 0$, i.e. we will define $FC^{n, 0, i}$ explicitly but for $FC^{0, k, i}$ we expect the reader to fill in $n = 0$ in the inductive definition below.\\
		
		We define for $i = 0, 1$:
		
		\[FC^{n}_{i} = \bigcup_{k = 0}^{\infty}FC^{n, k, 1}_{i}\]
		
		\[FC^{n, 0, 0}_{0} = (A^{n}\mathcal{R}^{0} - \id)\circ \ln^{n}\]
		
		\[FC^{n, 0, 0}_{1} = (A^{n}\mathcal{NC} - \id)\circ \ln^{n}\]
		
		Let $\Vec(.)$ denote the vector space generated by given elements. Suppose that $FC^{n, k, 0}_{i}$ is defined for $i = 0, 1$. Then we define:
		
		\[FC^{n, k, 1}_{0} = \Vec(FC^{n, k, 0}_{0})\]
		
		\[FC^{n, k, 1}_{1} = \Vec(FC^{n, k, 0}_{0}, FC^{n, k, 0}_{1})\]
		
		Suppose that $FC^{n, k, 1}_{i}$ is defined for $i = 0, 1$. Then $FC^{n, k + 1, 0}_{i}$ is defined by adding the items indicated by the axioms A4-A8, e.g. for A4b you would add:
		
		\[FC^{n, k, 1}_{i}\circ\left(\id + \left(FC^{n, k, 1}_{i}\circ A^{-n}(G^{n - 1}_{\preceq \id})\right)\right)\]
	\end{Definition}
	
	Then the sets $FC^{n}_{i}$ are clearly the smallest sets of functions satisfying axioms $A2 - A8$. All that is left is proving that the sets $FC^{n}_{i}$ also satisfy axiom $A1$.\\
	
	\begin{Remark}
		When we talk about proving something by structural induction, or by induction on the axioms we really mean by induction on the $k$ in the above construction.
	\end{Remark}
	
	In particular let us treat convergent polycycles. In the interest of the reader's time we will simply use the full power of \cite{Kaiser2017AnalyticCO} which in and of itself contains enough to prove Theorem \ref{ThmConvPoly}, without using Theorem \ref{ThmSuffAxFull}.
	
	\begin{Remark}
		Note that \cite{Kaiser2017AnalyticCO} may have enough to prove Theorem \ref{ThmConvPoly}, it does not have enough to immediately prove Proposition \ref{PropQAEst}. For which we do Theorem \ref{ThmSuffAxFull} and domain extensions in order to prove the upper and lower bound of $A1$.
	\end{Remark}
	
	\begin{Proposition}
		Take $(\mathcal{R}, \mathcal{NC})$ given by taking $\mathcal{R}$ to be the set of convergent near affine small Dulac series, i.e. any series:
		
		\[a + \sum_{i}P_{i}(\zeta)e^{-\lambda_{i}\zeta}\]
		
		with $a \in \Aff$, $P_{i}$ polynomials and $\lambda_{i} > 0$ and increasing to $+\infty$ which converges pointwise for $\re(\zeta) > C$ for some $C > 0$. Take $\mathcal{NC} = \{\id\}$.\\
		
		Then:
		
		\begin{enumerate}
			\item{All convergent polycycles are $(\mathcal{R}, \mathcal{NC})$-admissible.}
			\item{$FC^{n, k, i}_{0} = FC^{n, k, i}_{1}, k > 0$.}
			\item{For all elements of $FC^{n}_{i}$ holds that they either obey the lower bound of A1 or they are smaller than any exponential.}
		\end{enumerate}
	\end{Proposition}
	
	\begin{proof}
		It is not difficult to check that (see \cite[p. 19]{ilyashenkoFiniteness}) the Dulac map for the formal normal form:
		
		\[\frac{x^{k + 1}}{1 + a x^{k}}\frac{\partial}{\partial x} - y\frac{\partial}{\partial y}\]
		
		is given by:
		
		\[\Delta_{st}(z) = e^{\frac{1}{k}}e^{- \frac{1 - akz^{k}\ln(z)}{kz^{k}}} = e^{- \frac{1 - z^{k} - akz^{k}\ln(z)}{kz^{k}}}\]
		
		it is straightforward to calculate (see also \cite[p. 36]{ilyashenkoFiniteness}) that if:
		
		\[a(\zeta) = k\zeta - \ln(k)\]
		
		\[f(\zeta) = \zeta+\ln\left(1 - \frac{1 - a\left(\zeta + \ln(k)\right)}{k} e^{-\zeta}\right)\]
		
		that then:
		
		\[\Delta_{st}^{\log} = \exp \circ f \circ a \in \exp \circ \mathcal{R}^{0} \circ \Aff\]
		
		So any convergent polycycle is $(\mathcal{R}, \mathcal{NC})$-admissible for this choice (it is easy to check that Taylor series with zero constant term written in logarithmic coordinate gives something in $\mathcal{R}$ with all the $P_{i}$ even being constants).\\
		
		Another thing worth noting is that $\mathcal{R}$ is closed under composition because we can simply go back to the natural coordinate from the logarithmic coordinate and then we are in near identity $\mathbb{R}_{an, \exp}$-germs around the origin expandable in terms of (real powers of) $\ln(z)$ and $z$. To compose two of them you simply need to Taylor expand out infinitesimal term in the logarithm. Convergence is roughly implied by composition being well defined in \cite{Kaiser2017AnalyticCO}, so convergence in a radius gives the required domain back in the logarithmic chart.\\
		
		Then the second statement is by construction.\\
		
		The third statement comes from the observation that every element of $FC^{n}_{i}$ is in fact an element of $\mathbb{R}_{an, \exp}$ as in \cite{Kaiser2017AnalyticCO} thus it is in a large Hardy field together with all $e^{-\lambda x}$, so it is either larger than one of them of smaller than all of them.
	\end{proof}
	
	\begin{Remark}
		By point 2 of the above Proposition we will drop the subscripts in $FC^{n}_{i}$, $F^{n}_{ig}$ (defined the same as $\mathcal{F}^{n}_{ig}$ relative to $\mathcal{FC}^{n}_{i}$) and $H^{n}_{i}$.\\
		
		Note also that in this case $J^{n} = \{\id\}$ for all $n$. This essentially makes axiom A4c trivial in this case. So when we go over the axioms we will skip A4c.
	\end{Remark}
	
	The rest of the section will be dedicated to doing two things, proving the upper bound of $A1$ and proving that any element smaller than any exponential is in fact zero. For that we will use the following:
	
	\begin{Theorem}[{\cite[Lemma 24.37 (2)]{Ilyashenko08lectureson}}]\label{ThmPL2}
		Let $\mathbb{C}^{+}$ be the real positive half plane inside $\mathbb{C}$ (without the imaginary axis). Let $f$ be an analytic function on $\mathbb{C}^{+}$, continuous up to and bounded on the imaginary axis.\\
		
		Suppose $f$ is smaller than any exponential on the real axis, then it is identically zero.
	\end{Theorem}
	
	In order to use this we will need to know much more about the domains of $FC^{n, k}$.
	
	\subsection{Controllable functions}
	
	So we have now made a minimal set which allows for proving finiteness as in the previous section. There are two ways in which this minimality makes the functions significantly easier to handle, the first of which is what we will call controllability, this was used in \cite{ilyashenkoFiniteness}, but never given a name.
	
	\begin{Definition}
		Let $U \subset \mathbb{C}$ be open, let $h$ be an analytic function on $U$, we call $h$ controllable on $U$ if for all $\epsilon > 0$ and for all $z \in U$ with $\re(z)$ large enough (dependent on $\epsilon$):
		
		\[|h(z)| < e^{\re(z)^{\epsilon}} \quad, \quad \left|\frac{dh(z)}{dz}\right| < e^{\re(z)^{\epsilon}}\]
	\end{Definition}
	
	The reason we look at this notion is because it is particularly compatible with the action of $A^{-1}$. But before we get there we first need an auxiliary definition:
	
	\begin{Definition}\label{DefPolyBound}
		Let $U \subset \mathbb{C}^{+}$ we call $U$ polynomially bounded if there exists a $d$ real such that for all $z \in U$, $|\re(z)|$ large enough:
		
		\[|\im(z)| \leq |\re(z)|^{d}\]
	\end{Definition}
	
	We now get the following:
	
	\begin{Proposition}\label{Prop Off real est}
		Let $h$ be controllable on $U$, then $A^{-1}(h)$ is controllable on any polynomially bounded subset of $\exp(U)$, moreover for all $\epsilon > 0$ we have for $\re(z)$ large enough:
		
		\[\left|\frac{d\ln(A^{-1}(h))}{dz}(z)\right| \leq \frac{1}{\re(z)^{1 - \epsilon}}\]
	\end{Proposition}
	
	\begin{proof}
		Let $V$ be some polynomially bounded subset of $\exp(U)$, note then that also for some $d > 0$, we have for all $z \in V$:
		
		\[|z| \leq |\re(z)|^{d}\]
		
		Then $z \in V$, for $\epsilon > 0$ and $\re(z)$ large enough:
		
		\begin{align*}
			|A^{-1}(h(z))| & = |e^{h(\ln(z))}| \leq e^{|h(\ln(z))|} \leq e^{e^{\re(\ln(z))^{d\epsilon}}} \leq e^{e^{(\ln|z|)^{d\epsilon}}}\\
			& \overset{*}{\leq} e^{e^{d\epsilon\ln|z|}} \leq e^{|z|^{d\epsilon}} \leq e^{|\re(z)|^{d\epsilon}}\\
		\end{align*}
		
		At $*$ we use the fact that for $x$ real and large enough and $\epsilon$ small enough, $x^{\epsilon} \leq \epsilon x$. Then because $\epsilon$ is arbitrarily small, we are done. Next:
		
		\begin{align*}
			\left|\frac{dA^{-1}(h)}{dz}(z)\right| & = |A^{-1}(h(z))|\left|\frac{dh}{dz}(\ln(z))\right|\frac{1}{|z|}\\&\\
			& \leq\frac{e^{\re(z)^{d\epsilon}}e^{\re(\ln(z))^{d\epsilon}}}{|z|} \leq \frac{e^{\re(z)^{d\epsilon}}|z|^{d\epsilon}}{|z|} \leq e^{\re(z)^{d\epsilon}}
		\end{align*}
		
		Retaking the estimates above we have for $\re(z)$ large enough and $\epsilon < 1$:
		
		\[\left|\frac{d\ln\left(A^{-1}(h)\right)}{dz}(z)\right| = \left|\frac{dh(\ln(z))}{dz}\right| \leq \frac{1}{|z|^{1 - d\epsilon}} \leq \frac{1}{\re(z)^{1 - d\epsilon}}\]
		
		Note that $A^{-1}(h)$ is the exponential of something and thus non-zero, so we are done.
	\end{proof}
	
	Controllable functions are as the name might suggest in some sense particularly nice, in particular they are compatible with what we will call domains of almost degree $d$.
	
	\begin{Definition}
		A positive function $f: \mathbb{R} \to \mathbb{R}$ is called almost degree $d \in \mathbb{R}$ if for all $\epsilon$ and $x$ large enough, dependent on $\epsilon$:
		
		\[f(x) \geq x^{d - \epsilon}\]
		
		A $U \subset \mathbb{C}^{+}$ is almost degree $d$ if there exists some almost degree $d$ function $f$ such that for some $a > 0$:
		
		\[\{z \in \mathbb{C}\mid \re(z) > a, -f(\re(z)) < \im(z) < f(\re(z))\} \subset U\]
		
		Because we will need it more often an almost degree $1$ domain will be called an almost sector and an almost degree $0$ domain will be called an almost strip.
	\end{Definition}
	
	\begin{Remark}
		Being almost degree $d$ has no upper bound in how large either the function or the domain can be. However we are of course mostly interested in the case where some function is 'just barely' almost degree $d$, i.e. not even degree $d$, think of something like $f(x) = \frac{x}{\ln(x)}$ which is almost degree $1$.
	\end{Remark}
	
	The reason we introduce these two concepts together is that they are compatible in the following sense:
	
	\begin{Proposition}\label{PropControlFuncBound}
		Let $g = A^{-1}(h)$ with $h$ controllable on some domain $\ln(U)$, suppose that for each $z$ in $U$ the straight line from $z$ perpendicular to the real axis is in $U$ (the part between $z$ and the real axis).\\
		
		Let $d \in \mathbb{R}$, then for each $\varepsilon > 0$ there exists:
		
		\begin{enumerate}
			\item{Functions $f_{1}$ and $f_{2}$ of almost degree $1$ such that for:
				
				\[V \coloneqq \left\{z \in \mathbb{C} \mid |\im(z)| \leq \min\left(\frac{f_{1}(\re(z))\re(z)^{d}}{g(\re(z))}, f_{2}(\re(z))\right)\right\}\]
				
				we have for $z \in V$:
				
				\[|\im(g(z))| < \varepsilon \re(z)^{d}\]}
			\item{An almost degree $1 + d$ domain $V$ such that for $z \in V$:
				
				\[|\arg(g)(z)| < \varepsilon \re(z)^{d}\]}
		\end{enumerate}
	\end{Proposition}
	
	\begin{proof}
		What we know is that for any $\delta, \epsilon > 0$ and any $x$ real and large enough:
		
		\[\left|\ln\left(\frac{g(x + i\delta)}{g(x)}\right)\right| \leq \frac{\delta}{x^{1 - \epsilon}}\]
		
		This implies two different statements by considering both the real and imaginary parts of the logarithm and the fact that $g$ is a real analytic function:
		
		\[|g(x + i\delta)| \leq e^{\frac{\delta}{x^{1 - \epsilon}}}|g(x)| \quad, \quad |\arg(g(x + i\delta))| \leq \frac{\delta}{x^{1 - \epsilon}}\]
		
		combining both and the fact that clearly for $x$ large enough $\arg(x + i \delta) = \arctan\left(\frac{\delta}{x}\right)$ we get by the rudimentary estimate that $\sin(\arctan(x)) < x$:
		
		\[|\im(g(x + i\delta))| \leq \frac{\delta}{x^{1 - \epsilon}}e^{\frac{\delta}{x^{1 - \epsilon}}}|g(x)|\]
		
		Now the first statement is implied by this final inequality, indeed, taking $\delta = \min\left(\frac{x^{1 + d - 2\epsilon}}{|g(x)|}, x^{1 - 2\epsilon}\right)$ will let $\im(g(x + i\delta))$ eventually grow at worst as $x^{d -\epsilon}$, but $\epsilon$ is arbitrary. Now concretely, let $x_{n}$ be large enough for this inequality to hold for $\epsilon = \frac{1}{n}$ and such that: 
		
		\[\frac{\min\left(\frac{x^{1 + d - \frac{2}{n}}}{|g(x)|}, x^{1 - \frac{2}{n}}\right)}{x^{1 - \frac{1}{n}}}e^{\frac{\min\left(\frac{x^{1 + d - \frac{2}{n}}}{|g(x)|}, x^{1 - \frac{2}{n}}\right)}{x^{1 - \frac{1}{n}}}}|g(x)| \leq \varepsilon x^{d}\]
		
		Assume that $x_{1} < x_{2} < \cdots$. Then we can make the function $f_{1} = f_{2}$ which is $x^{1 - \frac{2}{n}}$ between $x_{n}$ and $x_{n + 1}$, this is certainly almost degree $1$ and we get the inequality we want in a domain as described.\\
		
		For the final statement we retake an earlier formula:
		
		\[|\arg(g(x + i\delta))| \leq \frac{\delta}{x^{1 - \epsilon}}\]
		
		So we can just proceed as in the previous part.
	\end{proof}
	
	\begin{Lemma}\label{LemmaExpDDom}
		Let $U$ be an almost degree $d$ domain:
		
		\begin{enumerate}
			\item{If $0 < d \leq 1$ then $\exp(U)$ contains $\mathbb{C}^{+}$ outside some compactum.}
			\item{If $d \leq 0$ then $\exp(U)$ contains a domain of almost degree $1$.}
		\end{enumerate}
	\end{Lemma}
	
	\begin{proof}
		Let us go one by one:
		
		\begin{enumerate}
			\item{Note that if $d > 0$, any function of nearly degree $d$ has to go to infinity, thus any domain of nearly degree $d$ has to contain an abitrarily wide strip, thus the exponential of it contains $\mathbb{C}^{+}$ up to some compactum.}
			\item{If $d \leq 0$, note that any domain of nearly degree $d$ has to contain for $a$ large enough:
				
				\[\{z \in \mathbb{C} \mid \re(z) > a > 0,  -\re(z)^{d - 1} < \im(z) < \re(z)^{d - 1}\}\]
				
				We can look where the exponential of this domain goes. Because this is eventually inside a strip, we may just look at the image of the boundary, i.e.:
				
				\[e^{x + ix^{d - 1}} = e^{x}\cos(x^{d - 1}) + ie^{x}\sin(x^{d - 1})\]
				
				By the usual arguments the cosine goes to $1$ and we will approximate the graph of the function:
				
				\[\phi(x) = x\sin(\ln(x)^{d - 1})\]
				
				Which roughly behaves like $x\ln(x)^{d - 1}$ so it is clearly almost degree $1$.}
		\end{enumerate}
	\end{proof}
	
	\subsection{Standard domains}
	
	An idea which might already be present from a glance at the axioms and trying to imagine the domains coming from them, is that we will in some sense be dealing with non-traditional domains of analyticity. However we still really want the domains to be large enough such that we have Phraghm\'en-Lindel\"of for exponentially decreasing functions, i.e. if a bounded function on this domain decreases faster than any exponential, it is identically zero. For our context the right notion is called a standard domain in \cite[\S 1.5, Definition 1, p.55]{ilyashenkoFiniteness}.
	
	\begin{Definition}
		A $\Omega \subset \mathbb{C}^{+}$ is called a standard domain if:
		
		\begin{enumerate}
			\item{$\Omega$ is symmetric w.r.t the real axis.}
			\item{There exists a injective complex analytic map $\rho: \Omega \to \mathbb{C}$ mapping reals to reals such that:
				
				\[\rho(\Omega) = \mathbb{C}^{+}\setminus K\]
				
				$K$ being some compactum
				
				\[\rho' = 1 + o(1)\]
				
				The small $o$-notation being in function of the absolute value of the argument.}
			\item{This map $\rho$ extends analytically to some $\Omega_{\delta}$, i.e. all point at distance $< \delta$ from $\Omega$, except possibly on a compact subset.}
		\end{enumerate}
	\end{Definition}
	
	\begin{Remark}
		In particular any bounded analytic function on a standard domain descending faster than any exponential on the real axis must be identically zero by a rudimentary application of Phraghm\'en-Lindel\"of.
	\end{Remark}
	
	Our next goal is to better understand the class of standard domains. Let us start with an alternative criterion coming from \cite[\S 5.4 B]{ilyashenkoFiniteness}: 
	
	\begin{Proposition}\label{PropSDEquiv}
		A domain $\Omega$ is standard iff there exists an injective real analytic mapping $\id + \phi: \ln(\Omega) \to \Pi$, with $\Pi$ the strip symmetric around the positive real axis with width $\pi$ such that:\\ 
		
		First, the image of $\id + \phi$ is $\Pi$ up to some compactum, second, as $\re(z)$ increases:
		
		\[\phi(z) \to 0, \phi'(z) \to 0\]
		
		such that $\rho \coloneqq A^{-1}(\id + \phi)$ extends in the same way as the definition of a standard domain.
	\end{Proposition}
	
	\begin{proof}
		It is clear that $\rho$ goes from $\Omega$ to $\exp(\Pi)$ which is $\mathbb{C}^{+}$ outside some compactum, it is invertible far enough by the fact the map $\id + \phi$ will be dominated by the identity. Moreover $\phi'$ going to zero implies that $\rho' = 1 + o(1)$.
	\end{proof}
	
	There is a single class of easy to describe standard domains:
	
	\begin{Lemma}\label{LemmaDegdSD}
		Let $U$ be an almost degree $d > 1$ domain, then $U$ contains a standard domain.
	\end{Lemma}
	
	\begin{proof}
		We will do this by explicitly giving a standard domain of maximally almost degree $d$ for all $d > 1$ and because any domain of almost degree $d'$ must contain all domains of maximally almost degree $d < d'$ we will be done.\\
		
		So consider the function:
		
		\[\rho(z) \coloneqq z + (z + 2)^{\frac{1}{d}}\]
		
		Because $d > 1$, $\frac{1}{d} - 1 < 0$ and then:
		
		\[\rho'(z) = 1 + \frac{1}{d}(z + 2)^{\frac{1}{d} - 1}\]
		
		And on $\mathbb{C}^{+}$:
		
		\[\left|\frac{1}{d}(z + 2)^{\frac{1}{d} - 1}\right| \leq \frac{(\re(z) + 1)^{\frac{1}{d} - 1}}{d} < 1\]
		
		Thus $\rho'$ is a local biholomorphism on $\mathbb{C}^{+}$. So now we only want to prove that $\rho$ is injective. We want to prove that for all $z, \tilde{z}$ different:
		
		\[|\rho(z) - \rho(\tilde{z})| > 0\]
		
		Now:
		
		\begin{align*}
			|\rho(z) - \rho(\tilde{z})| &\geq \left||z - \tilde{z}| - \left|(z + 2)^{\frac{1}{d}} - (\tilde{z} + 2)^{\frac{1}{d}}\right|\right|\\ 
			& = \left||z - \tilde{z}| - \left|\int_{\tilde{z}}^{z}(w + 2)^{\frac{1}{d} - 1}dw\right|\right|\\
			& = \left||z - \tilde{z}| - |z - \tilde{z}|\left|\int_{0}^{1}(\tilde{z} + t(z - \tilde{z}) + 2)^{\frac{1}{d} - 1}dt\right|\right|\\
		\end{align*}
		
		Now, by virtue of $z, \tilde{z} \in \mathbb{C}^{+}$:
		
		\[|\tilde{z} + t(z - \tilde{z}) + 2| \geq 2\]
		
		Thus:
		
		\[|\rho(z) - \rho(\tilde{z})| > \left(1 - 2^{\frac{1}{d} - 1}\right)|z - \tilde{z}|\]
		
		Thus $\rho$ is a biholomorphism and $\rho(\mathbb{C}^{+})$ is a standard domain. Let us now study its boundary, by the fact that $\rho$ is a biholomorphism, we know that the boundary is the image of the boundary, i.e. it is $\rho(it)$ for $t \in \mathbb{R}$ and:
		
		\begin{align*}
			\rho(it) &= it + \left(2 + it\right)^{\frac{1}{d}} \\
			&= (4 + t^{2})^{\frac{1}{2d}}\cos\left(\frac{\arctan\left(\frac{t}{2}\right)}{d}\right) + i\left(t + (4 + t^{2})^{\frac{1}{2d}}\sin\left(\frac{\arctan\left(\frac{t}{2}\right)}{d}\right)\right)
		\end{align*}
		
		We now want to know dominating behaviour, note that because $d > 1$, both the sine and cosine will tend to some non-zero value, respectively $\sin\left(\frac{\pi}{2d}\right)$ and $\cos\left(\frac{\pi}{2d}\right)$ so in asymptotic considerations we can ignore both of then, then the function after some rescaling looks like:
		
		\[(4 + t^{2})^{\frac{1}{2d}} + i\left(t + (4 + t^{2})^{\frac{1}{2d}}\right)\]
		
		Let us introduce a new variable $s = (4 + t^{2})^{\frac{1}{2d}}$, then $t = (s^{2d} - 4)^{\frac{1}{2}}$ and the boundary is clearly parametrized by:
		
		\[s + i\left[(s^{2d} - 4)^{\frac{1}{2}} + s\right]\]
		
		which is clearly maximally nearly degree $d$.
	\end{proof}
	
	\subsection{A consequence of the ordering theorem}
	
	Before we get to asymptotics of $FC^{n}_{i}$, the domains and the upper bound of $A1$ we will first need some information on asymptotics of elements of $G^{n}$ and some other Lemma's.
	
	\begin{Proposition}\label{PropGMGE}
		Suppose that $OT_{n}$ holds, then for all $g, h \in G^{n - 1}$: 
		
		\begin{enumerate}
			\item{The following limits exist, for all $k \leq n$: 
				
				\[L_{k}(g) \coloneqq \lim_{x \to +\infty}\frac{A^{-k}(g)(x)}{x} \in [0, +\infty]\]
				
				If $k = n$ we call this limit the generalized multiplier of $g$ and denote it $GM(g)$, if $k = n - 1$ we call this limit the generalized exponent of $g$ and denote it $GE(g)$.}
			\item{if $g \leq h$, then $L_{k}(g) \leq L_{k}(h)$.}
			\item{If $L_{k}(g) < \infty$, then $L_{k -1}(g) \leq 1$, if also $L_{k} > 0$ then $L_{k -1}(g) = 1$.}
			\item{We have $0 < L_{k}(g\circ h^{-1}) < \infty$ if and only if $g \sim_{k} h$.}
		\end{enumerate}
	\end{Proposition}
	
	\begin{proof}
		Let us go point by point:
		
		\begin{enumerate}
			\item{In part $2$ we have proven that:
				
				\[A^{n}(\Aff), A^{n - 1}(\Aff) \subset G^{n - 1}\]
				
				So because applying $A^{-1}$ preserves order, it is clear by $OT_{n}$ implying that $G^{n - 1}$ is totally ordered, that:
				
				\[L_{k} = \sup\{C > 0 \mid A^{k}(C\id) \leq g\}\]}
			\item{Trivial by the previous alternative definition for $L_{k}$.}
			\item{Let $\lambda = L_{k} > 0$, then there exists some function $\epsilon$ going to zero such that:
				
				\[A^{-k}(g)(x) = \lambda x + \epsilon(x)x\]
				
				Applying $A$ to both sides gives that:
				
				\[A^{-(k - 1)}(g)(x) = x + \ln(\lambda) + \ln(1 + \epsilon(x))\]
				
				So it is clear that $L_{k - 1} = 1$, then by the second point the case $\lambda = 0$ follows.}
			\item{By definition $L_{k} = \lambda$ with $0 < \lambda < \infty$ means that:
				
				\[\lim_{x \to \infty}\frac{A^{-k}(g \circ h^{-1})(x)}{x} = \lambda\]
				
				precomposing with $\exp^{k}$ this says that:
				
				\[\lim_{x \to \infty}\frac{\left[\exp^{k}\circ (g \circ h^{-1})\right](x)}{\exp^{k}(x)} = \lambda\]
				
				precomposing with $h$ gives:
				
				\[\lim_{x \to \infty}\frac{\exp^{k}(g(x))}{\exp^{k}(h(x))} = \lambda\]
				
				which implies that $g \sim_{k} h$. Conversely take $g, h \in G^{n - 1}$ such that $g \sim_{k} h$, then it is clear that it is not possible that the limit for $L_{k}(g\circ h^{-1}) = 0, +\infty$ by exactly the same reasoning above.}
		\end{enumerate}
	\end{proof}
	
	\begin{Remark}
		The interest in studying these limits come from axiom $A4$ and $A5$ where one needs to precompose $\phi \in FC^{n}_{i}$ with elements $A^{-n}(g) \in A^{-n}(G^{n - 1})$, so let $U$ be the domain on which $\phi$ is defined, then we need to have some control over the domain on which $\phi \circ A^{-n}(g)$ is defined. Assuming for a second that $g$ is controllable, which we will prove later, then we can approximates what happens to a domain $U$ under the map $A^{-n}(g)$ by the values of $A^{-n}(g)$ on the real axis using Proposition \ref{PropControlFuncBound}.\\
		
		Note that if $GM(g) \in (0, +\infty)$, then on the real axis:
		
		\[A^{-n}(g)(x) = (GM(g) + o(1))x\]
		
		And if $GE(g) \in (0, +\infty)$, then on the real axis:
		
		\[A^{-n}(g)(x) = x^{GE(g) + o(1)}\]
		
		In principle the existence of the limit $L_{k}$ says something about how much oscillation there exists in the function, for example note that:
		
		\[\lim_{x \to \infty}\frac{(2 + \sin(x))x}{x}\]
		
		does not exist, however should we apply $A$ to this function we will get:
		
		\[\ln\left(2 + \sin\left(e^{x}\right)\right) + x\]
		
		and the limit of this divided by $x$ will exist. So the existence of $L_{k}(g)$ for high $k$ says something about the size of possible oscillations in $g$. Note that $G^{n - 1} \subset G^{n}$, so for $g \in G^{n}$ all $L_{k}(g)$ will have been shown to exist at the end.
	\end{Remark}
	
	\subsection{Main result}
	
	\begin{Proposition}\label{PropositionUpperBoundA1}
		Suppose that for $m \leq N$, all $FC^{m}$ satisfy $A1$ (on the real line). Then for all $n \leq N$:
		
		\begin{enumerate}
			\item[$1_{n}$]{\begin{enumerate}
					\item{For every element $\alpha$ of $FC^{n}$ there exists a $g \in G^{n - 1}_{\sim\id}$ such that $\alpha$ can be analytically continued to $A^{-n}(g^{-1})(\Omega)$ for some standard domain $\Omega$.}
					\item{All $\alpha \in FC^{n}$, up to any polynomially bounded (Definition \ref{DefPolyBound}) subset of the domains as described in the first point, admits an upper bound of the form:
						
						\[|\alpha(\zeta)| \leq e^{-\lambda \re(\zeta)}\]}
			\end{enumerate}}
			\item[$2_{n}$]{For every element $g \in G^{n}$ there exsts a compactum $K$ such that $g$ is controllable on some $h^{-1}\ln^{n}(\Omega)$, $\Omega$ a domain of almost degree $d > 1$ and $h \in G^{n - 1}$. Even more, on that domain there exists a $C > 0$ such that:
				
				\[|g(z)| \leq C\re(z), |g'(z)| \leq C\]
				
				Note that $G^{n}$ is closed under inverses thus we get a similar lower bound.}
			\item[$3_{n}$]{For every $g \in G^{n}_{\preceq\id}$ and every standard domain $\Omega$, $A^{-(n + 1)}(g)(\Omega)$ also contains a standard domain.}
		\end{enumerate}
		
		Finally in this case $1_{N + 1}$ also holds.
	\end{Proposition} 
	
	\begin{proof}
		We will prove this by induction on $N$ and proceed by induction on $n$.\\ 
		
		Now all of $2_{-1}$, $2_{0}$, $3_{-1}$ and $3_{0}$ (using Dulac series and approximating by linear functions) are obvious, we will proceed by induction along the following scheme:
		
		\[1_{m} + 2_{m} + 3_{m}, (m < n) \Rightarrow 1_{n}\]
		
		\[1_{m} + 2_{m} + 3_{m}, (m < n) + 1_{n} \Rightarrow 2_{n}\]
		
		\[2_{n} \Rightarrow 3_{n}\]
		
		The crucial part is that the second implication will use $ADT_{n}$ thus we only get up to $3_{n}$ but the first step needs nothing so we get to $1_{N + 1}$.\\
		
		The first part is the hardest so let us do the second and third first:\\
		
		\underline{$1_{m} + 2_{m} + 3_{m}, (m < n) + 1_{n} \Rightarrow 2_{n}$}:\\
		
		Take a $g \in G^{n}$, by Proposition \ref{PropSupplAxioms} we can use ADT by our assumption. So we can write our $g$ as:
		
		\[g = a + \phi_{1} + \cdots + \phi_{v}\]
		
		where $a \in \Aff$, $\phi_{i} \in \mathcal{F}^{m_{i}}_{\bar{f}_{i}}$ in such a way that each $m_{i} \leq n - 1$ and:
		
		\[(m_{i}, \bar{f}_{i}) < (m_{i + 1}, \bar{f}_{i + 1})\]
		
		So let $\alpha \in FC^{m}_{1}$ we are then interested in the domain of $\alpha \circ \exp^{m} \circ f$ or equivalently we are interested in the domain of $\alpha \circ A^{-m}(f)$.\\ 
		
		Let us discuss the domain of $\alpha \circ A^{-m}(f)$, now by $1_{m}$, $\alpha$ itself has as domain with good estimates some polynomially bounded subset of $A^{-m}(h_{1}^{-1})(\mathbb{C}^{+})$ for some $h_{1} \in G^{m - 1}$. Moreover by $2_{m}$, $A^{-m}(f)$ has the domain $A^{-m}(h_{2}^{-1})(\Omega_{1})$ for some domain $\Omega_{1}$ of almost degree $> 1$ and some $h_{2} \in G^{m - 1}$. We can now look at the part of the domain of $\alpha$ where this composition is well defined and it is:
		
		\[A^{-m}(h_{1}^{-1})(\mathbb{C}^{+}) \cap A^{-m}(f \circ h_{2}^{-1})(\Omega_{1})\]
		
		which by $3_{m}$ contains some $A^{-m}(h^{-1})(\Omega_{2})$, again $\Omega_{2}$ a domain of almost degree $> 1$ and $h \in G^{m - 1}$. So the domain of $\alpha \circ A^{-m}(f)$ includes $A^{-m}(f^{-1}\circ h^{-1})(\Omega_{2})$.\\
		
		Note that we can use Lemma \ref{LemmaExpDDom} to now control domains in the following sense, Let $m = n - 1$ then we can combine everything into a domain we want by statement $3_{n - 1}$, for $m < 0$ we started with a standard domain and we take essentially some exponentials of it, so by Lemma \ref{LemmaExpDDom} we end up with $\mathbb{C}^{+}$ at most outside some compactum, we can ignore all those contributions.\\
		
		Because we have an upper estimate of $(f^{-1}\circ h^{-1})'$ by Proposition \ref{Prop Off real est} and controllability of $2_{m}$, we know by Cauchy estimates with circles of radius $e^{-x^{\frac{1}{2}}}$ that $\alpha$ and its derivative go to zero exponentially. This gives both the upper bound for $g$ and $g'$ that we need.\\ 
		
		So let us move on to the second step:\\
		
		\underline{$2_{n} \Rightarrow 3_{n}$}\\
		
		Let us write for some $g \in G^{n}$:
		
		\[A^{-m}(g) = \id + \phi_{m}\]
		
		it is easy to prove using Proposition \ref{PropControlFuncBound} that for any strip $\Pi$, $\im(\phi_{m})$ goes to zero on $\ln^{n - m}(\Pi)$, $m \leq n$ (Note that this is not an optimal domain).\\
		
		From this and the fact that:
		
		\[\phi_{m}(z) = z\cdot(e^{\phi_{m - 1}\circ \ln} - 1)\]
		
		we can show that $\arg(\phi_{n})$ goes to zero or $\pi$ depending on which one is larger as $|z|$ grows on $\exp(\Pi)$.\\
		
		From this we can conclude using Proposition \ref{PropGMGE} that if $A^{-(n + 1)}(g) = o(1)x$ on $\exp(\Pi)$ and thus on $\Pi$:
		
		\[A^{-n}(g) = x + \ln(o(1))\]
		
		with the imaginary part of $\ln(o(1))$ going to zero, the rest follows from Proposition \ref{PropSDEquiv}.\\
		
		Otherwise we have much the same with: $A^{-(n + 1)}(g) = Cx(1 + o(1))$ and thus:
		
		\[A^{-n}(g) = x + \ln(C(1 + o(1)))\]
		
		So let us move on to the final step.\\
		
		\underline{$1_{m} + 2_{m} + 3_{m}, (m < n) \Rightarrow 1_{n}$}:\\
		
		We will prove the case $FC^{n}$ by proving all three statements for all $FC^{n, k, 0}$ by induction on $k$.\\
		
		\underline{\textbf{Base of induction: $k = 0$}}:\\
		
		Let us go item by item:
		
		\begin{enumerate}
			\item[Domains]{By definition:
				
				\[FC^{n, 0, 0} = (A^{n}\mathcal{R}^{0} - \id)\circ\ln^{n}\]
				
				Now an element of $\mathcal{R}^{0}$ can be analytically continued to some shifted version of $\mathbb{C}^{+}$ then $A^{n}$ of this element is defined on $\ln^{n}(\Omega)$, subtracting the identity makes no difference in domain, then precomposing with $\ln^{n}$ makes everything defined on a shifted $\mathbb{C}^{+}$ again, which certainly contains a standard domain and then we can take the $g$ in the definition the identity.}
			\item[Asymptotics]{With $FC^{n, 0, 0}$ one starts with the Dulac series from $\mathcal{R}^{0}$. Note then that:
				
				\[A\left(\zeta \mapsto \zeta + \zeta^{h}e^{-\lambda \zeta}\right) = \ln\left(e^{\zeta} + e^{h\zeta}e^{-\lambda e^{\zeta}}\right) = \zeta + \ln\left(1 + e^{(h - 1)\zeta}e^{-\lambda e^{\zeta}}\right)\]\\
				
				Done inductively one sees that applying $A^{n}$ to a near identity Dulac series gives $\id$ plus series in the variables $\exp(\zeta), ..., \exp^{n + 1}(\zeta)$, every term having a factor $\exp^{n + 1}(\zeta)$. Subtracting $\id$ and precomposing as in $FC^{n, 0, 0}$ with $\ln^{n}$ gives series in the variables $\ln^{n - 1}(\zeta), ..., \exp(\zeta)$, every term having a factor $\exp(\zeta)$. So we can use the fact that Dulac series give asymptotic approximations to get an upper bound for $FC^{n, 0, 0}$ exactly as in $A1$ (again see \cite{Kaiser2017AnalyticCO} for a more formal version).}
		\end{enumerate}
		
		\underline{\textbf{Induction step: done for all $h < k$, $m < n$}}:\\
		
		It is clear that none of the three things we need to prove are affected by sums or scalar multiplication so we only need to deal with going from $FC^{n, k - 1, 1}_{i}$ to $FC^{n, k, 0}_{i}$, i.e. elements added by the axioms $A4-A7$.\\
		
		\underline{\textbf{Axiom $A7$}}:\\
		
		It is clear that $A7$ can never alter domain, so let us consider asymptotics:\\ 
		
		Clearly multiplying or dividing by some iterated logarithm will not matter. Let $\alpha \in FC^{n, k - 1, 1}_{i}$ then there are two things for which we need to find an exponential upper bound:
		
		\[|\ln(1 + \alpha)| \quad, \quad |e^{\alpha} - 1|\]
		
		The basic principle is the same, develop into Taylor series and start estimating, in fact upon ignoring the coefficients for the Taylor series of the exponential both can be estimated by:
		
		\[\sum_{n \geq 1}|\alpha(x)|^{n} = \frac{|\alpha|}{1 + |\alpha(x)|}\]
		
		which is clearly exponentially descending.\\
		
		\underline{\textbf{Axiom $A6$}}:\\
		
		Axiom $A6$ will be dealt with in the following fashion, let $f$ be an invertible analytic function, note then that by Cauchy's Theorem:
		
		\[f^{-1}(w_{0}) - w_{0} = \frac{1}{2\pi i}\int\frac{f^{-1}(w) - w}{w - w_{0}}dw\]
		
		making the substitution $w = f(z)$ we get:
		
		\[f^{-1}(w_{0}) - w_{0} = \frac{1}{2\pi i}\int\frac{(z - f(z))f'(z)}{f(z) - w_{0}}dz\]
		
		Let us now take $\alpha \in FC^{n, k - 1, 1}_{i}$ and take $f = \id + \alpha$ we get:
		
		\[(\id + \alpha)^{-1}(w_{0}) - w_{0} = \frac{1}{2\pi i}\int\frac{\alpha(z)(1 + \alpha'(z))}{z + \alpha'(z) - w_{0}}dz\]
		
		noting that performing Cauchy estimates on circles of radius $e^{-\re(z)^{\frac{1}{2}}}$ will still preserve our domains and we get both the domain and the estimates we want. We only need to prove that $\id + \alpha$ is invertible for $\re(z)$ large enough.\\ 
		
		Let $f(z) = z + \alpha(z)$, suppose there exists an almost degree $0$ or $1$ function $\tilde{f}$ such that for all $x \geq x_{0}$ $\alpha$ is defined on the line between $(x, \tilde{f}(x))$ and $(x, -\tilde{f}(x))$, let us then consider $\tilde{f}(z) = f(z + 2x_{0}) - 2x_{0}$, $\tilde{\alpha} = \alpha(z + 2x_{0})$ note that by definition:
		
		\[\tilde{f}(z) = z + \tilde{\alpha}(z)\]
		
		but unlike with $\alpha$, $\tilde{\alpha}^{i}$, i.e. the $i$-fold composition makes sense and we can explicitly give the inverse of $\tilde{f}$ as:
		
		\[\tilde{f}^{-1}(w) = \sum_{i = 0}^{\infty}\tilde{\alpha}^{i}(w)\]
		
		this implies that $z \mapsto \tilde{f}(z) - 2x_{0} = f(z + 2x_{0})$ will eventually be injective and because $f'$ will remain nonzero by Cauchy estimates we know $f$ will eventually be injective on the type of domain that is necessary (noting that the $\tilde{f}$ stays, i.e. we do not need to make the domain smaller or anything, i.e. in the case of $i = 0$ we actually do get what we want).\\
		
		\underline{\textbf{Axiom $A4$ Domain}}:\\
		
		Axiom $A4a$ is the right domain by construction.\\
		
		Axiom $A4b$ is the same story as $A4a$ only we now taylor expand the composition to get as domain the intersection of the two domains of $FC^{n}$ and for this we use $3_{n}$.\\
		
		\underline{\textbf{Axiom $A4$ Asymptotics}}:\\
		
		With axiom $A4a$ it is clear by Proposition \ref{PropGMGE} that one precomposes with nearly a linear function, so an exponential upper bound is kept.\\
		
		Axiom $A4b$ is even easier, by Proposition \ref{PropSupplAxioms}, any element in $G^{n - 1}$ goes to infinity and so does an element of $A^{-n}(G^{n - 1})$, so here we precompose with something near identity.\\
		
		\underline{\textbf{Axiom $A5$}}:\\
		
		Let $\alpha \in FC^{n, k - 1, 1}$, $\phi$ and $\psi$ be as in axiom $A5$.\\
		
		Let us first discuss the easiest one, axiom $A5c$, here we are interested in:
		
		\[(e^{\alpha} - 1)e^{\psi}\]
		
		The domain requirement is easy because an element of $G^{n - 1}\ln^{n}$ has as domain $\mathbb{C}$ outside of some compactum. and because $\psi \in \exp^{p}G^{n - 1}\ln^{n}$, $p < n$, at worst $e^{\psi}$ will grow slower than any exponential whereas:
		
		\[|e^{\alpha(x)} - 1| = \left|\int_{0}^{\alpha(x)}e^{y}dy\right| \leq |\alpha(x)|e^{\alpha(x)}\]
		
		which is certainly exponentially small.\\ 
		
		We now treat $A5a$ and $A5b$ together. The domain is obvious, note that in both cases the domain of $\psi$ is $\mathbb{C}^{+}$ minus some compactum. We claim that $\psi$ grows at most linearly on the real axis. Indeed the case $A5a$ follows from the fact that  $\psi \in \exp^{p}G^{n - 1}\ln^{n}$, $p < n$ so at worst $e^{\psi}$ will grow slower than any exponential, implying that $\psi$ tself grows at most linearly. In the case of $A5b$ this follows from Proposition \ref{PropGMGE}, so we get the domains we want by one final application of $3_{n}$.\\ 
		
		We move on to asymptotics. Note that
		
		\[\phi(\psi(x) + \alpha(x)) - \phi(\psi(x)) = \int_{\psi(x)}^{\psi(x) + \alpha(x)}\phi'(y)dy\]
		
		By the usual Cauchy estimates $\phi'$ is exponentially small and $\alpha$ is assumed exponentially small. So because $\psi$ goes to infinity the entirety will eventually be exponentially small.
	\end{proof}
	
	\begin{Corollary}
		Let $\alpha$ be in some $FC^{n}$, and let $\alpha$ be smaller than any exponential on the real axis, then it is identically zero.\\
		
		In particular this implies Proposition \ref{PropQAEst}.
	\end{Corollary}
	
	\begin{proof}
		The first statement is obviously a combination of point $1$ and $3$ of the previous Proposition combined with Theorem \ref{ThmPL2}.\\
		
		The second is then a combination of ADT with Proposition \ref{PropSupplAxioms}.
	\end{proof}
	
	\section{Remark}
	
	\subsection{A historical note}
	
	As we heavily base our proof on the original proof in \cite{ilyashenkoFiniteness}, let us also give some indication of what we have taken (in most cases almost verbatim) from \cite{ilyashenkoFiniteness}.\\ 
	
	The Structural Theorem and its proof are taken from the proof of Ilyashenko (cf. \cite[\S 1.3]{ilyashenkoFiniteness}) though the final Structural Theorem came from notes of Ilyashenko communicated in private to the author.\\
	
	The large induction proving sufficiency of axioms comes almost verbatim from the proof of Ilyashenko (cf. \cite[\S 1.8-1.11, 2.4, 2.7-2.8, 2.10-2.11]{ilyashenkoFiniteness}), the difference is that instead of relying on the results of \cite[\S 2.3, 2.5]{ilyashenkoFiniteness} we instead replace them by axioms turning everything into formal manipulations. The argument for proving the Ordering Theorem come from the notes communicated to the author\\
	
	The idea of looking at the smallest set of functions making these formal manipulations possible and allowing us to arrive at the Additive Decomposition Theorem is new to this article. This allows us to clearly separate the bookkeeping, i.e. the sufficiency of axioms, from the geometry, i.e. proving that $FC^{n}_{i}$ satisfies $A1$.\\
	
	The parts about standard domains are taken straight from the proof of Ilyashenko (cf. \cite[\S 5.4B]{ilyashenkoFiniteness}).\\
	
	The main technical Lemmas are also take straight from Ilyashenko (cf. \cite[\S 5.4B Proposition 5, \S 1.4C Proposition 1, Proposition 2]{ilyashenkoFiniteness}) and the way the main lemma's are used, which is scattered throughout \cite[Chapter 5]{ilyashenkoFiniteness}. The name of controllable functions is new.\\ 
	
	The idea of a domain of almost degree $d$ is new, though the way it is used is certainly inspired by arguments like \cite[\S 5.2C assertion $*$]{ilyashenkoFiniteness}.\\
	
	The upper bound of $A1$ and analytic continuation are all standard, however it is certainly parallel to much of the arguments in proving regularity and extendibility of functions gotten through Shift Lemma's in \cite[Chapter II, Chapter IV]{ilyashenkoFiniteness}.\\
	
	\subsection{The origin of the subscripts 0 and 1}\label{SubsecOrigin}
	
	In this article we have presented $\mathcal{R}$ and $\mathcal{NC}$ in some sense as parameters which determine which polycycles you are able to describe, but this was certainly not the case in \cite{ilyashenkoFiniteness}. Let us briefly discuss what was the original intention in \cite{ilyashenkoFiniteness}.\\
	
	Originally $\mathcal{R}$ was the almost regular functions also found in \cite[Definition 24.27]{Ilyashenko08lectureson}, essentially it is in the logarithmic chart some germ which extends to a special set of domains of almost degree 2 called standard quadratic domains which also admit a (not necessarily convergent) Dulac series of the form:
	
	\[\zeta + \sum_{k}P_{k}(\zeta)e^{-\beta_{j}\zeta}\]
	
	The reason these were chosen is because they can describe the transit maps around hyperbolic saddles, see \cite[Proposition 24.39]{Ilyashenko08lectureson}.
	
	Originally $\mathcal{NC}$ are the cochains in for example \cite{MartinetRamisSemihyperbolic} which are able to normalize a semihyperbolic saddle, for some more elaboration on its effects, see \cite{YeungCounterExample}. But for now all we want to remark is that elements of $\mathcal{NC}$ are only able to be analytically continued to a strip (preserving Dulac series asymptotics).\\
	
	The set of admissible polycycles corresponding to this (the framework needs to be altered slightly to account for complex values, but it can be done) is all polycycles, so proving axiom A1 in this context would prove the entire Theorem of Dulac. This, or more appropriately the reverse of it, was basically the approach in \cite{ilyashenkoFiniteness}. There, a set of functional cochains $\mathcal{FC}^{n}_{i}$ was defined a priori and then \cite{ilyashenkoFiniteness} tried to prove that all the axioms hold.\\
	
	In our more constructive (in the classical sense) approach we can in fact use Proposition \ref{PropControlFuncBound} to show that an element of $FC^{n}_{0}$ can be analytically continued to a domain of almost degree $1$ while elements in $FC^{n}_{1}$ can only be analytically continued to a domain of almost degree $0$ and this is the reason why they are split, exactly in order to have that $A^{-(n + 1)}G^{n}$ (which appears in some of the axioms) is defined on $\mathbb{C}^{+}$ outside some compactum. If we were to make no distinction between $FC^{n}_{0}$ and $FC^{n}_{1}$ then we would only be able to get the domain of $A^{-(n + 1)}G^{n}$ to be something of almost degree $1$.
	
	\section{Bibliography}
	
	\bibliography{mybib}

\begin{thebibliography}{1}

\bibitem{ecalleconstructiveproof}
J.~Ecalle.
\newblock {\em Introduction aux fonctions analysables et preuve constructive de
  la conjecture de Dulac}.
\newblock Actualit{\'e}s math{\'e}matiques. Hermann, 1992.

\bibitem{ilyashenkoFiniteness}
Yu.~S. Ilyashenko.
\newblock {\em Finiteness theorems for limit cycles}, volume~94 of {\em
  Translations of Mathematical Monographs}.
\newblock American Math. Society, 1991.

\bibitem{Ilyashenko2002CentennialHO}
Yulii Il'yashenko.
\newblock Centennial history of hilbert’s 16th problem.
\newblock {\em Bulletin of the American Mathematical Society}, 39:301--354,
  2002.

\bibitem{Ilyashenko08lectureson}
Yulij Ilyashenko and Sergei Yakovenko.
\newblock Lectures on analytic differential equations.
\newblock In {\em Graduate Studies in Mathematics}. American Mathematical
  Society, 2008.

\bibitem{Kaiser2017AnalyticCO}
Tobias Kaiser and Patrick Speissegger.
\newblock Analytic continuations of log-exp-analytic germs.
\newblock {\em Transactions of the American Mathematical Society}, 2017.

\bibitem{MartinetRamisSemihyperbolic}
Jean Martinet and Jean-Pierre Ramis.
\newblock Probl\`emes de modules pour des \'equations diff\'erentielles non
  lin\'eaires du premier ordre.
\newblock {\em Publications Math\'ematiques de l'IH\'ES}, 55:63--164, 1982.

\bibitem{Moussu1991}
Roche~C. Moussu, R.
\newblock Théorie de hovanskii et problème de dulac.
\newblock {\em Inventiones mathematicae}, 105(2):431--442, 1991.

\bibitem{YeungCounterExample}
Melvin {Yeung}.
\newblock {On the monograph ``Finiteness Theorems for limit cycles'' and a
  special case of alternant cycles}.
\newblock {\em arXiv e-prints}, page arXiv:2402.12506, February 2024.

\end{thebibliography}
	\bibliographystyle{plain}

\end{document}